\documentclass[12pt, reqno]{amsart}

\usepackage{amsmath, amsthm, amssymb, stmaryrd}
\usepackage{hyperref}
\usepackage{enumerate}
\usepackage{url}

\usepackage{bm}

\usepackage{xcolor}

\usepackage[centering, margin={1in, 0.5in}, includeheadfoot]{geometry}

\usepackage{fancyvrb}

\newtheorem{thm}{Theorem}[section]
\newtheorem{lemma}[thm]{Lemma}
\newtheorem{prop}[thm]{Proposition}
\newtheorem{cor}[thm]{Corollary}

\theoremstyle{definition}

\theoremstyle{remark}
\newtheorem{remark}[thm]{Remark}

\numberwithin{equation}{section}

\DeclareMathOperator{\Mod}{mod}
\newcommand{\mmod}[1]{\;(\Mod{ #1})}

\def\alp{{\alpha}} 
\def\bet{{\beta}}  
 
\def\del{{\delta}} 
 
\def\tet{{\theta}}  
\def\kap{{\kappa}}
\def\lam{{\lambda}} \def\Lam{{\Lambda}}

\def\sig{{\sigma}}

\def\ome{{\omega}}  
\def\eps{\varepsilon}

\def\le{\leqslant} \def\ge{\geqslant}

\def\d{{\,{\rm d}}}

\def \sig{{\sigma}}
\def \cont {{\mathrm{cont}}}

\def \bC {\mathbb C}

\def \bN {\mathbb N}

\def \bQ {\mathbb Q}
\def \bR {\mathbb R}
\def \bZ {\mathbb Z}
\def \bT {\mathbb T}

\def \bn {{\mathbf{n}}}

\def \bu {\mathbf u}
\def \bv {\mathbf v}
\def \bx {\mathbf x}

\def \by {\mathbf y}
\def \bz {\mathbf z}

\def \balp {\boldsymbol{\alp}}

\def \fm {\mathfrak m}
\def \fn {\mathfrak n}

\def \fM {\mathfrak M}

\def \cA {\mathcal A}
\def \cB {\mathcal B}
\def \cC {\mathcal C}
\def \cD {\mathcal D}

\def \cF {\mathcal F}

\def \cP {\mathcal P}
\def \cQ {\mathcal Q}

\def \cS {\mathcal S}

\def \ord {\mathrm{ord}}

\def \Li {{\mathrm{Li}}}

\def \N {\mathbb{N}}

\def \R {\mathbb{R}}
\def \T {\mathbb{T}}
\def \Z {\mathbb{Z}}

\def \balpha {\bm{\alpha}}
\def \btheta {\bm{\theta}}

\newcommand{\str}{\mathrm{str}}
\newcommand{\sml}{\mathrm{sml}}
\newcommand{\unf}{\mathrm{unf}}

\begin{document}
	\title[Arithmetic Ramsey theory over the primes]{Arithmetic Ramsey theory over the primes}
	\author{Jonathan Chapman \and Sam Chow}
	
	\address{School of Mathematics, University of Bristol, Bristol, BS8 1UG, UK, and the Heilbronn 
		Institute for Mathematical Research, Bristol, UK}
	\email{jonathan.chapman@bristol.ac.uk}
	\address{Mathematics Institute, Zeeman Building, University of Warwick, Coventry CV4 7AL, United Kingdom}
	\email{Sam.Chow@warwick.ac.uk}
	
	\maketitle
	
	\begin{abstract} 
		We study density and partition properties of polynomial equations in prime variables. We consider equations of the form $a_1h(x_1) + \cdots + a_sh(x_s)=b$, where the $a_i$ and $b$ are fixed coefficients, and $h$ is an arbitrary integer polynomial of degree $d$. Provided there are at least $(1+o(1))d^2$ variables, we establish necessary and sufficient criteria for this equation to have a monochromatic non-constant solution with respect to any finite colouring of the prime numbers. We similarly characterise when such equations admit solutions over any set of primes with positive relative upper density. In both cases, we obtain counting results which provide asymptotically sharp lower bounds for the number of monochromatic or dense solutions in primes. Our main new ingredient is a uniform lower bound on the cardinality of a prime polynomial Bohr set.
	\end{abstract}
	
	\setcounter{tocdepth}{1}
	\tableofcontents
	
	\section{Introduction}
	
	An influential theorem of Szemer\'edi asserts that sets of positive integers $A$ with positive upper density, meaning that
	\begin{equation*}
		\limsup_{N\to\infty}\frac{|A\cap\{1,2,\ldots,N\}|}{N}>0,
	\end{equation*}
	must contain arbitrarily long arithmetic progressions. Green and Tao \cite{GT2008} famously established a version of Szemer\'edi's theorem for the primes. Specifically, writing $\cP:=\{2,3,5,\ldots\}$ for the set of prime numbers, Green and Tao showed that sets $A\subseteq\cP$ satisfying
	\begin{equation}\label{eqn1.1}
		\limsup_{N\to\infty}\frac{|A\cap\{1,2,\ldots,N\}|}{|\cP\cap\{1,2,\ldots,N\}|}>0
	\end{equation}
	contain arbitrarily long arithmetic progressions. In particular, the primes themselves contain arithmetic progressions of any finite length.
	
	One can consider configurations other than arithmetic progressions. We call a system of Diophantine equations \emph{density regular} if it has non-constant solutions over all sets of positive integers with positive upper density. For example, consider a linear homogeneous equation
	\begin{equation}
		\label{eqn1.2}
		a_1 x_1 + \cdots + a_s x_s = 0,
	\end{equation}
	where $s\geqslant 3$ and $a_1,\ldots,a_s$ are non-zero integers. Roth \cite{Rot1954} showed that this equation is density regular if and only if $a_1+\cdots +a_s = 0$. Green \cite{Gre2005A} subsequently proved that Roth's theorem holds over the primes; the equation (\ref{eqn1.2}) has non-constant solutions over any set of primes $A\subseteq\cP$ satisfying (\ref{eqn1.1}) if and only if $a_1+\cdots +a_s = 0$.
	
	A related, weaker notion of regularity is that of \emph{partition regularity}, which refers to systems of Diophantine equations which admit monochromatic non-constant solutions with respect to any finite colouring of the positive integers. By the pigeonhole principle, density regularity implies partition regularity. A foundational result in arithmetic Ramsey theory is \emph{Rado's criterion} \cite[Satz IV]{Rad1933}, which completely characterises partition regularity for finite systems of linear equations. In particular, Rado's criterion reveals that Equation (\ref{eqn1.2}) is partition regular if and only if there exists a non-empty set $I\subseteq\{1,\ldots,s\}$ such that $\sum_{i\in I}a_i = 0$.
	
	L\^e \cite{Le2012} observed that Green and Tao's work provides a characterisation of partition regularity for systems of linear homogeneous equations in shifted primes. For single equations, L\^e proved that the equation (\ref{eqn1.2}) admits monochromatic non-constant solutions with respect to any finite colouring of $\cP+1:=\{p+1:p\in\cP\}$ or $\cP-1$ if and only if there exists a non-empty set $I\subseteq\{1,\ldots,s\}$ such that $\sum_{i\in I}a_i = 0$. Note that there are divisibility obstructions which prevent such a result holding over the set $\cP+c$ for integers $c\notin\{-1,1\}$. For example, the equation $x+y=z$ is partition regular, but if we partition $\cP+c$ into residue classes modulo $q$ for any prime $q$ dividing $c$, then there are no monochromatic solutions to $x+y=z$.
	
	The purpose of this paper is to obtain a complete classification of partition and density regularity over primes for equations in sufficiently many variables of the form
	\begin{equation}\label{eqn1.3}
		a_1 h(x_1) + \cdots + a_s h(x_s) = b,
	\end{equation}
	where $a_1,\ldots,a_s$ are non-zero integers, $b$ is an integer, and $h$ is a polynomial with integer coefficients. We say that (\ref{eqn1.3}) is \emph{partition regular over the primes} if every finite colouring of the prime numbers produces a monochromatic non-constant solution to (\ref{eqn1.3}). Similarly, we call (\ref{eqn1.3}) \emph{density regular over the primes} if (\ref{eqn1.3}) has a non-constant solution over any set of primes $A\subseteq\cP$ which satisfies (\ref{eqn1.1}). For $b=0$, observe that L\^e's result \cite{Le2012} asserts that Rado's condition characterises partition regularity over primes for (\ref{eqn1.3}) whenever $h(x)=x\pm 1$.
	
	In our previous work \cite{CC}, we established necessary and sufficient conditions for partition and density regularity (over $\N$) for all equations (\ref{eqn1.3}) in sufficiently many variables. We observed that it is necessary for partition regularity that $h$ satisfies a certain `intersectivity condition' in order to avoid divisibility obstructions, as alluded to above. For partition regularity over primes, we use the following definition, previously introduced in \cite{Le2014, Rice2013}. An integer polynomial $h$ is \emph{intersective of the second kind} if for each positive integer $n$, there exists an integer $x$ which is coprime to $n$ such that $n$ divides $h(x)$. Observe that any polynomial $h$ satisfying $h(1)=0$ or $h(-1)=0$ is intersective of the second kind. However, one can construct numerous polynomials, such as $(x^2 - 13)(x^2 - 17)(x^2 - 221)$ and $(x^3 - 19)(x^2 + x + 1)$, which are intersective of the second kind but do not have rational zeros (see \cite[\S2]{LS2015}).
	
	Our main result is the following.
	
	\begin{thm}\label{thm1.1}
		Let $d\geqslant 2$ be an integer. There exists a positive integer $s_0(d)$ such that the following is true. 
		Let $h$ be an integer polynomial of degree $d$, and let $s\geqslant s_0(d)$ be an integer. Let $a_1,\ldots,a_s$ be non-zero integers, and let $b$ be an integer.
		\begin{enumerate}
			\item[(PR)]  Equation (\ref{eqn1.3}) is partition regular over the primes if and only if there exists a non-empty set $I\subseteq\{1,\ldots,s\}$ with $\sum_{i\in I}a_i =0$ and an integer $m$ with $b=(a_{1}+\cdots +a_s)m$ such that $h(x) - m$ is an intersective polynomial of the second kind.
			\item[(DR)] Equation (\ref{eqn1.3}) is density regular over the primes if and only if $b=a_1 +\cdots +a_s = 0$.
		\end{enumerate}
		Furthermore, we have $s_0(2)=5$, $s_0(3)\leqslant 9$, and
		\begin{equation}\label{eqn1.4}
			s_0(d) \leqslant d^2-d+2\lfloor \sqrt{2d+2} \rfloor + 1 \qquad (d\geqslant 4).
		\end{equation}
	\end{thm}
	
	\begin{remark}
		The integer $s_0(d)$, which is defined explicitly in \S\ref{sec2}, was previously introduced in \cite{CC} to study partition and density regularity of equations (\ref{eqn1.3}).
	\end{remark}
	
	\begin{remark}
		The conditions in the partition regularity statement are trivially satisfied when $b=a_1+\cdots +a_s = 0$. Indeed, in this case, we can take $I=\{1,\ldots,s\}$ and $m=h(1)$. It then follows that $h(x) - m$ is the zero polynomial, which is trivially intersective of the second kind.
	\end{remark}
	
	In our previous work \cite{CC}, we obtained asymptotically sharp lower bounds for the number of solutions to (\ref{eqn1.3}) lying in dense or monochromatic subsets of the positive integers. Our second main result is the following analogous counting version of Theorem \ref{thm1.1} for subsets of $\cP_N:=\{p\leqslant N: p \text{ prime}\}$.
	
	\begin{thm}\label{thm1.4}
		Let $d\geqslant 2$ be an integer, and let $s_0(d)$ be as given in Theorem \ref{thm1.1}. 
		Let $h$ be an integer polynomial of degree $d$. Let $s\geqslant s_0(d)$ be an integer. Let $a_1,\ldots,a_s$ be non-zero integers, and let $b$ be an integer. Given a set of integers $\cA$, write
		\begin{equation*}
			S(\cA) = \{(x_{1},\ldots,x_s)\in\cA^s : x_i\neq x_j \text{ for all } i\neq j, \text{ and }
			a_1h(x_1) + \cdots + a_s h(x_s) =0\}.
		\end{equation*}
		\begin{enumerate}
			
			\item[(PR)]  Suppose there exists a non-empty set $I\subseteq\{1,\ldots,s\}$ with $\sum_{i\in I}a_i =0$ and an integer $m$ with $b=(a_{1}+\cdots +a_s)m$ such that $h(x) - m$ is an intersective polynomial of the second kind. Then for any positive integer $r$ there exists a positive real number $c_1 = c_1(h;a_1,\ldots,a_s,b;r)$ and a positive integer $N_1 = N_1(h;a_1,\ldots,a_s,b;r)$ such that the following is true for any positive integer $N\geqslant N_1$. Given any $r$-colouring $\cP_N = \cC_1 \cup\cdots\cup \cC_r$,
			there exists $k\in\{1,\ldots,r\}$ such that
			$|S(\cC_k)|\geqslant c_1N^{-d}(N/\log N)^{s}$.
			
			\item[(DR)] If $a_1 +\cdots +a_s = b = 0$, then for any positive real number $\delta>0$ there exists a positive real number $c_2= c_2(h;a_1,\ldots,a_s;\delta)$ and a positive integer $N_2 = N_2(h;a_1,\ldots,a_s;\delta)$ such that the following is true for any positive integer $N\geqslant N_2$. Given any set $A\subseteq\cP_N$ satisfying $|A|\geqslant\delta |\cP_N|$, we have $|S(A)|\geqslant c_2N^{-d}(N/\log N)^{s}$.
		\end{enumerate}
	\end{thm}
	
	\subsection*{Methods}
	Our goal is to find many monochromatic/dense solutions to
	\[
	L_1(h(\bx)) = L_2(h(\by))
	\]
	for some linear forms $L_1$ and $L_2$, where $L_1(1,\ldots,1) = 0$. In Fourier space, after normalisation and accounting for small prime moduli (the $W$-trick), the image of $h$ can be shown to behave like $\bN$ for our count. The upshot is that it suffices to count solutions to an equation
	\[
	L_1(\bn) = L_2(h_D(\bz)),
	\]
	where $h_D$ is a related polynomial that is intersective of the second kind. This Fourier-analytic transference principle was introduced by Green \cite{Gre2005A} to show that relatively dense sets of primes contain three-term arithmetic progressions, and is based on Fourier decay and restriction (from harmonic analysis). The transference argument is sketched in further detail in the next section, and formalised in the five sections afterwards. It can be regarded as a version of the Hardy--Littlewood circle method, and estimates for prime Weyl sums feature prominently in our work.
	
	To count monochromatic/dense solutions to our linearised equation, we use an arithmetic regularity lemma. This enables us to decompose the indicator functions of our colour classes into three parts, the first of which exhibits quasi-periodic structure and ultimately dominates the count. Using this quasi-periodicity to obtain a large count requires us to show that polynomials evaluated at primes are dense in Bohr sets, in a suitably uniform sense. The colour class that we choose maximises this density.
	
	Our main novelty is to uniformly bound from below the density of a `prime polynomial Bohr set'. This is accomplished by induction on the dimension, beginning with a result of Harman \cite{Har1993} from Diophantine approximation. The latter brings prime Weyl sums into play. The requisite data for these comes partly from the analysis in the earlier sections, partly from Lucier's pioneering work on intersective polynomials \cite{Luc2006}, and partly from the investigations of L\^e--Spencer \cite{LS2014} and Rice \cite{Rice2013} into polynomials that are intersective of the second kind.
	
	\subsection*{Organisation}
	
	We begin in \S\ref{sec2} with some preliminary results. We prove the `only if' parts of Theorem \ref{thm1.1} by establishing necessary conditions for (\ref{eqn1.3}) to be partition or density regular over the primes. We synthesise both the density and partition statements presented in Theorem \ref{thm1.4} into a single result, Theorem \ref{thm2.5}, on counting solutions to certain linear form equations. We also recall the `auxiliary intersective polynomials' of Lucier \cite{Luc2006} and use them to state Theorem \ref{thm2.7}, which is a `linearised' version of Theorem \ref{thm2.5}. Finally, we provide a sketch of the subsequent transference argument we will use to deduce Theorem \ref{thm2.5} from Theorem \ref{thm2.7}. 
	
	In \S\ref{sec3}, we introduce formally the `linearisation' procedure which we will use to infer Theorem \ref{thm2.5} from Theorem \ref{thm2.7}. We apply the $W$-trick and introduce the majorant $\nu$, the latter of which is the focus of our investigations in \S\ref{sec5} and \S\ref{sec6}.
	To expedite this process, we record in \S\ref{sec4} some general results on exponential sums over primes. These will later be used in \S\ref{sec5} and \S\ref{sec9}.
	
	In \S\ref{sec5}, we study the Fourier transform of our majorant $\nu$ by using the Hardy--Littlewood circle method. We follow this in \S\ref{sec6} by investigating the restriction properties of $\nu$ and a related majorant $\mu_D$. The Fourier decay and restriction estimates we obtain in these sections are then applied in \S\ref{sec7} to execute the transference principle. This completes the deduction of Theorem \ref{thm2.5} from Theorem \ref{thm2.7}.
	
	The focus of the final two sections is to prove Theorem \ref{thm2.7} using an arithmetic regularity lemma. In \S\ref{sec8}, we begin this argument by first modifying Theorem \ref{thm2.7} into a new result (Theorem \ref{thm8.1}) which is more amenable to arithmetic regularity methods.  This reduces matters to counting primes in `polynomial Bohr sets'. Finally, in \S\ref{sec9}, we prove Theorem \ref{thm8.1} by establishing density estimates for these prime polynomial Bohr sets. 
	
	\subsection*{Notation} 
	Let $\N$ denote the set of positive integers, and write $\cP :=\{2,3,5,7,11,\ldots\}$ for the set of prime numbers. For each prime $p$, let $\bQ_p$ and $\bZ_p$ denote the $p$-adic numbers and the $p$-adic integers respectively. Given a real number $X>0$, we write $[X] := \{n\in\N:n\leqslant X\}$ and $\cP_X :=\cP\cap[X]$.
	Set $\bT = [0,1]$. For each $d\in\N$ and $\balpha=(\alpha_1,\ldots,\alpha_d)\in\R^d$, we define
	\begin{equation*}
		\lVert \balpha\rVert := \max_{1\leqslant i\leqslant d}\min_{n\in\Z}|\alpha_i - n| = \min_{\bn\in\Z^d}\lVert \balpha - \bn\rVert_\infty.
	\end{equation*}
	For $q \in \bN$ and $x \in \bR$, we write $e(x) = e^{2 \pi i x}$ and $e_q(x) = e(x/q)$. 
	For $h(x) \in \bZ[x]$ and $\bx = (x_1,\ldots,x_s)$, where $s \in \bN$, we abbreviate $h(\bx) = (h(x_1),\ldots,h(x_s))$. 
	If $L$ is a polynomial with integer coefficients, we write $\gcd(L)$ for the greatest common divisor of its coefficients. The letter $\eps$ denotes a small, positive constant, whose value is allowed to differ between separate occurrences. We employ the Vinogradov and Bachmann--Landau asymptotic notations.
	In any statement in which $\eps$ appears, we assert that the statement holds for all sufficiently small $\eps>0$.
	For a finitely supported function $f: \bZ \to \bC$, the \emph{Fourier transform} $\hat{f}$ is defined by
	\begin{equation*}
		\hat f(\alp) := \sum_{n \in \bZ} f(n) e(\alp n) \qquad (\alp \in \bR).   
	\end{equation*}
	Given a real-valued function $G$ which is bounded over a closed interval $[a,b]$, we write
	\begin{equation*}
		\lVert G\rVert_{L^\infty[a,b]}:=\sup_{a\leqslant t\leqslant b}|G(t)|.
	\end{equation*}
	If $G$ is continuously differentiable on an open interval containing $[a,b]$, then we define
	\begin{equation*}
		\lVert G\rVert_{\cS[a,b]} := \lVert G\rVert_{L^{\infty}[a,b]} + \lVert(b-a) G'\rVert_{L^{\infty}[a,b]}.
	\end{equation*}
	
	\subsection*{Acknowledgements} 
	JC is supported by the Heilbronn Institute for Mathematical Research. SC thanks the University of Bristol for their kind hospitality on various occasions when this work was being discussed.
	
	\section{Preliminaries}
	\label{sec2}
	
	\subsection{Useful ingredients}
	
	We will make repeated use the Siegel--Walfisz theorem \cite[Lemma 7.14]{Hua1965}, which we now state for convenience. Recall that the logarithmic integral is given by
	\[
	\Li(x) = \int_2^x \frac{\d t}{\log t} \qquad (x \ge 2).
	\]
	
	\begin{thm} 
		[Siegel--Walfisz theorem]
		\label{thm2.1}
		Let $P \ge 2$, and write $L = \log P$. Let $A > 0$. Then there exists $c = c(A) > 0$ such that the following is true. Let $n,q,a \in \bZ$ with
		\[
		1 \le n \le P, \qquad
		1 \le q \le L^A.
		\]
		Then
		\[
		\# \{ p \le n: p \equiv a \mmod q \} = \frac{\Li(n)}{\varphi(q)}
		+ O(Pe^{-c \sqrt{\log L}}).
		\]
	\end{thm}
	
	We will also make repeated use of \cite[Lemma 9]{Rice2013}, which we state below. Owing to an inaccuracy in the published version of Rice's paper, we cite a later arXiv version. This is also explained in a remark immediately following \cite[Lemma 9]{Rice2013}. The fact that $C$ only depends on the degree comes from the proof.
	
	\begin{lemma}
		\label{lem2.2}
		For any integer $k\ge 2$, there exists $C = C(k) > 0$ such that the following holds. Let $g(x) = a_k x^k + \cdots + a_1 x + a_0 \in \bZ[x]$, and let $W, b \in \bZ$. Let $a \in \bZ$ and $q \in \bN$ be coprime. Let $q = q_1 q_2$, where $q_2$ is the greatest divisor of $q$ that is coprime to $W$, and let $\cont(g) = \gcd(a_1,\ldots,a_k)$. Then 
		\[
		\left|
		\sum_{\substack{\ell=0\\(W\ell+b,q)=1}}^{q-1}
		e_q(ag(\ell))
		\right| \le C^{\ome(q)}
		\left(
		\gcd(\cont(g),q_1) \gcd(a_k,q_2)
		\right)^{1/k}
		q^{1-1/k}.
		\]
	\end{lemma}
	
	\subsection{Necessary conditions}
	
	We now provide necessary conditions for Equation (\ref{eqn1.3}) to be partition or density regular over the primes. To state our results, we recall that an integer polynomial $h$ is called \emph{intersective} (or \emph{intersective of the first kind)} if for each $n\in\N$, there exists $x\in\Z$ such that $h(x)\equiv 0\mmod{n}$. We call $h$ \emph{intersective of the second kind} if this statement holds under the additional condition that such an $x$ can be found which is coprime to $n$.
	The following lemma demonstrates that intersectivity is a necessary condition for partition regularity for general polynomial equations.
	
	\begin{lemma}
		\label{lem2.3}
		Let $s\in\N$ and let $F\in\Z[x_1,\ldots,x_s]$. Consider the equation
		\begin{equation}
			\label{eqn2.1}
			F(x_1,\ldots,x_s) = 0.
		\end{equation}
		\begin{enumerate}
			\item[(PR)] If (\ref{eqn2.1}) is partition regular (over the primes), then the single-variable polynomial $F(x,\ldots,x)\in\Z[x]$ is intersective (of the second kind).
			\item[(DR)] If (\ref{eqn2.1}) is density regular or density regular over the primes, then $F(x,\ldots,x)$ is the zero polynomial.
		\end{enumerate}
	\end{lemma}
	
	\begin{proof}
		Suppose (\ref{eqn2.1}) is partition regular. Let $n\in\N$ and consider the $n$-colouring of $\N$ defined by partitioning $\N$ into distinct residue classes modulo $n$. The existence of a monochromatic solution to (\ref{eqn2.1}) with respect to this colouring implies that $F(t,\ldots,t)\equiv 0 \mmod{n}$ holds for some $t\in[n]$. As $n$ was arbitrary, it follows that $F(x,\ldots,x)$ is intersective.
		
		Now suppose (\ref{eqn2.1}) is partition regular over the primes. Let $n\in\N$. As before, we partition into residue classes modulo $n$ and infer the existence of $t\in[n]$ and primes $p_1,\ldots,p_s$, which are not all equal, with $p_1\equiv \ldots \equiv p_s\equiv t\mmod{n}$ such that $F(p_1,\ldots,p_s)=0$. If we take $n$ to be a prime power, then, since the $p_i$ are not all equal, at least one $p_j$ is coprime to $n$, whence $t$ and $n$ are coprime. Applying the Chinese remainder theorem, we conclude that $F(x,\ldots,x)$ is intersective of the second kind.
		
		Finally, suppose (\ref{eqn2.1}) is density regular or density regular over the primes. Let $m\in\N$. By the Siegel--Walfisz theorem (in the case of density regularity over the primes), for each prime $p\nmid m$, we can find an integer solution $(x_1,\ldots,x_s)$ to (\ref{eqn2.1}) with $x_1\equiv \dots \equiv x_s\equiv m\mmod{p}$. By reducing (\ref{eqn2.1}) modulo $p$, we deduce that $F(m,\ldots,m)$ is divisible by infinitely many primes, whence $F(m,\ldots,m)=0$. As $m$ was arbitrary, we conclude that $F(x,\ldots,x)$ is the zero polynomial.
	\end{proof}
	
	We now apply this lemma to (\ref{eqn1.3}) to establish the `only if' parts of Theorem \ref{thm1.1}. By working modulo $|\mu| n$ for any $n \in \bN$, we see that if $\mu \ne 0$ and $\mu h$ is intersective of the second kind then so too is $h$. Note also that the following result does not impose any restriction on the number of variables.
	
	\begin{cor}\label{cor2.4}
		Let $s\in\N$ and let $h$ be an integer polynomial of positive degree. Let $a_1,\ldots,a_s$ be non-zero integers, and let $b$ be an integer.
		\begin{enumerate}
			\item[(PR)] If (\ref{eqn1.3}) is partition regular over the primes, then there exists a non-empty set $I\subseteq\{1,\ldots,s\}$ with $\sum_{i\in I}a_i =0$ and an integer $m$ with $b=(a_{1}+\cdots +a_s)m$ such that $h(x) - m$ is an intersective polynomial of the second kind.
			\item[(DR)] If (\ref{eqn1.3}) is density regular or density regular over the primes, then $b=a_1 +\cdots +a_s = 0$.
		\end{enumerate}
	\end{cor}
	
	\begin{proof}
		Throughout this proof, we write $\mu =a_1+\cdots+a_s$ and $H(x) =\mu h(x) - b$.
		First suppose (\ref{eqn1.3}) is partition regular over the primes. Applying Lemma \ref{lem2.3}, we find that $H(x)$ is intersective of the second kind. By considering solutions to $H(x)\equiv 0\mmod{d}$ for any $d\mid \mu$, we observe that $\mu\mid b$. If $\mu\neq 0$, then there is a unique $m\in\Z$ with $b=\mu m$ such that $H(x)=\mu(h(x) - m)$, whence $h(x)-m$ is intersective of the second kind. If $\mu=0$, then $b=0$ and so, upon taking $m=h(1)$, we have $b=\mu m$ and $h(x)-m$ is trivially intersective of the second kind. In both cases, Equation (\ref{eqn1.3}) becomes
		\begin{equation*}
			\sum_{i=1}^{s} a_i(h(x_i) - m) = 0,
		\end{equation*}
		for some $m\in\Z$ such that $h(x)-m$ is intersective of the second kind. As this new equation is partition regular, we infer from \cite[Proposition 2.1]{CC} the existence of a set $I\subseteq\{1,\ldots,s\}$ with the desired properties.
		
		Finally, suppose that (\ref{eqn1.3}) is density regular or density regular over the primes. Then Lemma \ref{lem2.3} implies that $H(x) = \mu h(x) - b$ is the zero polynomial. Since $h$ has positive degree, we conclude that $b=\mu=0$.
	\end{proof}
	
	In view of these necessary conditions, Theorem \ref{thm1.1} is now an immediate consequence of Theorem \ref{thm1.4}.
	
	% \begin{proof}[Proof of Theorem \ref{thm1.1} given Theorem \ref{thm1.4}]
		% The `only if' statements follow immediately from Corollary \ref{cor2.4}. The remaining `if' statements may be inferred from Theorem \ref{thm1.4}.
		% \end{proof}
	
	\subsection{Linear form equations}
	
	Having dispensed with the necessary conditions for partition and density regularity, we focus on finding monochromatic or dense solutions to (\ref{eqn1.3}). The necessary conditions we have established therefore inform us that (\ref{eqn1.3}) takes the shape
	\begin{equation*}
		\sum_{i\in I}a_i (h(x_i)-m) = -\sum_{j\in[s]\setminus I}a_j (h(x_j)-m),
	\end{equation*}
	where $I\subseteq [s]$ is non-empty with $\sum_{i\in I}a_i=0$, and $m\in\bZ$ is such that $b=(a_1+\cdots+a_s)m$ and $h(x)-m$ is intersective of the second kind. Upon replacing $h(x)$ with $h(x)-m$, we can therefore reduce to the case where $b=0$ and $h(x)$ is intersective of the second kind.
	
	To find monochromatic or dense solutions to (\ref{eqn1.3}) with $b=0$, we study equations of the form
	\begin{equation}\label{eqn2.2}
		L_1(h(\bx)) = L_2(h(\by)),
	\end{equation}
	for some linear forms $L_1$ and $L_2$. To avoid trivialities, we only consider non-degenerate  linear forms, where $L(\bx)=a_1x_1 + \cdots + a_sx_s$ is \emph{non-degenerate} if $a_i\neq 0$ for all $i\in[s]$. For this new equation, the necessary conditions for partition and density regularity become $L_1(1,\ldots,1)=0$. Following the recent works \cite{CC,CLP2021}, we address both density and partition regularity for (\ref{eqn2.2}) simultaneously by seeking solutions where the $x_i$ variables are sourced in a dense subset of $\cP_X$ whilst the remaining $y_j$ variables come from a colour class $\cC_k\subseteq \cP_X$.
	
	Before proceeding to our results, we require some notation. We begin by providing an explicit description of the threshold $s_0(d)$ for the number of variables required in our main theorems.
	Let  $T = T(d) \in \bN$ be minimal such that if $h(x) \in \bZ[x]$ has degree $d$, then
	\[
	h(x_1) + \cdots + h(x_T) = h(x_{T+1}) + \cdots + h(x_{2T})
	\]
	has $O_{h,\eps}(X^{2T - d + \eps})$ solutions $\bx \in [X]^{2T}$. Equivalently, by orthogonality, $T = T(d)$ is the smallest positive integer such that the moment estimate
	\begin{equation*}
		\int_{\T}\left\lvert \sum_{x\leqslant X}e(\alpha h(x))\right\rvert^{2T}\ll_{h,\eps} X^{2T-d+\eps}
	\end{equation*}
	holds for any integer polynomial $h$ of degree $d$. The quantity $s_0(d)$ appearing in Theorem~\ref{thm1.1} is now defined to be $s_0(d):= 2T(d) + 1$. 
	
	It follows from Hua's lemma \cite[Equation (1)]{Hua1938} that $T(2) \le 2$ and $T(3) \le 4$. In general, the proof of \cite[Corollary 14.7]{Woo2019} delivers
	\[
	T(d) \le \frac{d(d-1)}2 + \lfloor \sqrt{2d+2} \rfloor.
	\]
	These observations verify the bound (\ref{eqn1.4}) for $s_0(d)$.
	Finally, by considering solutions with $x_{i}=x_{i+T}$ for $i=1,2,\ldots,T$, we record the lower bounds
	\begin{equation}\label{eqn2.3}
		T(d) \geqslant d,
		\qquad
		s_0(d) \ge 2d + 1.
	\end{equation}
	
	We can now state our main result on partition and density regularity over primes for linear form equations (\ref{eqn2.2}).
	
	\begin{thm}
		\label{thm2.5}
		Let $r$ and $d\geqslant 2$ be positive integers, and let $0<\delta<1$. 
		Let $h$ be an integer polynomial of degree $d$ which is intersective of the second kind.
		Let $s \ge 1$ and $t \ge 0$ be integers such that $s + t \ge s_0(d)$. Let
		\[
		L_1(\bx) \in \bZ[x_1,\ldots,x_s], \qquad L_2(\by) \in \bZ[y_1,\ldots,y_t] 
		\]
		be non-degenerate linear forms such that $L_1(1,\ldots,1) = 0$. Then there exist 
		\[
		X_0=X_0(\delta,h,r,L_1,L_2)\in\N,
		\qquad
		\tau_0(\delta)=\tau_0(h,r,L_1,L_2;\delta)\in (0,1)
		\]
		such that the following is true for all $X\geqslant X_0$. Suppose $\cP_X = \cC_1 \cup \cdots \cup \cC_r$.
		Then there exists $k \in [r]$ with $|\cC_k|\geqslant\tau_0(\delta) |\cP_X|$
		such that if $A \subseteq \cP_X$ satisfies
		$|A| \ge \del |\cP_X|$, then 
		\[
		\# \{ (\bx, \by) \in A^s \times \cC_k^t: L_1(h(\bx)) = L_2(h(\by)) \} \gg 
		\frac{X^{s+t-d}}{(\log X)^{s+t}}.
		\]
		The implied constant may depend on $h, L_1, L_2, r, \delta$. 
	\end{thm}
	
	\begin{remark}
		In the case $t=0$, we have a linear form $L_2$ in zero variables and we are counting solutions $\bx\in A^s$ to the equation
		\begin{equation*}
			L_1(h(\bx)) = 0.
		\end{equation*}
		Note that when $t=0$, all linear forms $L_2$ in $t$ variables are vacuously non-degenerate.
	\end{remark}
	
	By harnessing a combinatorial `cleaving' argument of Prendiville \cite{Pre2021}, we can swiftly deduce Theorem \ref{thm1.4} from Theorem \ref{thm2.5}.
	
	\begin{proof}[Proof of Theorem \ref{thm1.4} given Theorem \ref{thm2.5}]
		Following the argument given at the beginning of this subsection, we may reduce to the case where $b=0$ and $h$ is intersective of the second kind. Combining \cite[Lemma 3.2]{CC} with \eqref{eqn2.3}, for $N$ sufficiently large, the number of solutions $\bx \in [N]^s$ to (\ref{eqn1.3}) such that $x_i = x_j$ holds for some $i \neq j$ is $O_\eps(N^{s-d+\eps-1/2})$. Therefore, by setting $t=0$, we see that the density regularity statement in Theorem~\ref{thm1.4} follows directly from Theorem~\ref{thm2.5}. Similarly, given a colouring $\cP_N = \cC_1\cup\cdots\cup\cC_r$, provided $N$ is sufficiently large, it remains to show that there are at least $c_1N^{-d}(N/\log N)^{s}$ monochromatic solutions to (\ref{eqn2.2}) with
		\begin{equation*}
			L_1(\bx):=\sum_{i\in I}a_ix_i, \quad\text{and} \quad L_2(\by):= -\sum_{i\in [s]
				\setminus I}a_iy_j.
		\end{equation*}
		
		For each $\delta>0$, let $\tau_0(\delta)\in(0,1)$ be as given in the statement of Theorem \ref{thm2.5}. By making minor adjustments, we may assume that $\tau_0(\delta)\leqslant\delta$. Set $\delta_0=1/r$, and for each $i\in[r]$ let $\delta_i:=\tau_0(\delta_{i-1})$, whence $0<\delta_r\leqslant\ldots\leqslant\delta_0<1$. Take $N\geqslant X_0(\delta_r,h,r,L_1,L_2)$ as in Theorem~\ref{thm2.5}, and suppose $\cP_N
		=\cC_1\cup\cdots\cup\cC_r$. For each $0\leqslant i\leqslant r$, let $k_i\in[r]$ be the index given by applying Theorem~\ref{thm2.5} with $\delta=\delta_i$. By the pigeonhole principle, we can find $k\in[r]$ and $0\leqslant i<j\leqslant r$ such that $k_i=k_j=k$. Therefore 
		\[
		\frac{|\cC_k|}{|\cP_N|}
		\ge \tau_0(\delta_i)
		=\delta_{i+1}\geqslant\delta_j,
		\]
		and
		\[
		\# \{ (\bx, \by) \in A^{|I|} \times \cC_k^{s-|I|}: L_1(h(\bx)) = L_2(h(\by)) \} \gg_{\delta,h,r,L_1,L_2} 
		\frac{N^{s-d}}{(\log N)^{s}}
		\]
		holds for any $A\subseteq\cP_N$ with $|A|\geqslant\delta_j$. Taking $A=\cC_k$ finishes the proof.
	\end{proof}
	
	\subsection{Auxiliary intersective polynomials}\label{subsec2.4}
	
	The next step of our argument is to use a version of Green's Fourier-analytic transference principle \cite{Gre2005A} to obtain solutions to (\ref{eqn2.2}) by `transferring' solutions from a `linearised' equation. To make this precise, we first need to introduce the auxiliary intersective polynomials of Lucier \cite{Luc2006} which emerge during the execution of this process.
	
	Let $h$ be an integer polynomial of positive degree which is intersective of the second kind. Thus, for each prime $p$, we can find a $p$-adic unit $z_p \in \bZ_p^{\times}$ such that $h(z_p)=0$. Throughout this paper, we fix a choice of $z_p$ for each prime $p$ and let $m_p$ be the multiplicity of $z_p$ as a zero of $h$. We can then define the completely multiplicative function
	\begin{equation}\label{eqn2.4}
		\lam(D) := \prod_p p^{m_p \ord_p(D)} \qquad (D\in\N).
	\end{equation}
	By \cite[Equation (73)]{Luc2006}, we have
	\begin{equation}\label{eqn2.5}
		D \mid \lam(D) \mid D^d.
	\end{equation}
	Denote by $r_D$ the unique integer in the range $(-D,0]$ which satisfies
	\[
	r_D \equiv z_p \mmod{p^{\ord_p(D)} \bZ_p}
	\]
	for all primes $p$. As $h$ is intersective of the second kind, we note that
	$(r_D, D) = 1$.  
	
	Finally, with this notation in place, we define the auxiliary intersective polynomial
	\[
	h_D(x) := \frac{h(r_D + Dx)}{\lam(D)} \in \bZ[x].
	\]
	These polynomials and the surrounding notation were introduced by Lucier \cite{Luc2006}, who also showed that $h_D$ is indeed a polynomial with integer coefficients \cite[Lemma 21]{Luc2006}.
	The most important property of these auxiliary polynomials is that the greatest common divisor of the non-constant coefficients of $h_D$ is bounded uniformly in $D$. Specifically, for all $D\in\N$, \cite[Lemma 28]{Luc2006} states that
	\[
	\gcd(h_D - h_D(0)) \ll_h 1.
	\]
	As in \cite[\S6]{CC}, this bound is crucial to our investigation of exponential sums involving intersective polynomials (see \S\ref{sec7} and \S\ref{sec9}).
	
	We can now state our linearised version of Theorem \ref{thm2.5}.
	
	\begin{thm} \label{thm2.7}
		Let $r$ and $d\geqslant 2$ be positive integers, and let $0<\delta<1$. 
		Let $h$ be an integer polynomial of degree $d$ which is intersective of the second kind.
		Let $s \ge 1$ and $t \ge 0$ be integers such that $s + t \ge s_0(d)$. Let
		\[
		L_1(\bx) \in \bZ[x_1,\ldots,x_s], \qquad L_2(\by) \in \bZ[y_1,\ldots,y_t] 
		\]
		be non-degenerate linear forms such that $L_1(1,\ldots,1) = 0$.
		Then there exists 
		\[
		\eta = \eta(d,\delta,L_1,L_2) \in (0,1)
		\]
		such that the following is true.
		Let $D, Z \in \bN$ satisfy $Z \ge Z_0(D, h, r, \del, L_1, L_2)$, and set $N:=h_D(Z)$. 
		Suppose
		\[
		[\eta Z,Z]\cap\{ z \in [Z]: r_D + Dz \in \cP \}
		= \cC_1 \cup \cdots \cup \cC_r.
		\]
		Then there exists $k \in [r]$ such that if $\cA\subseteq[N]$ satisfies $|\cA|\geqslant\delta N$, then
		\[
		\# \{ (\bn,\bz) \in \cA^s \times \cC_k^t: L_1(\bn) = L_2(h_D(\bz)) \} \gg 
		N^{s-1} 
		\left(\frac{DZ}{\varphi(D)\log Z} \right)^t.
		\]
		The implied constant may depend on $h, L_1, L_2, r, \delta$.
	\end{thm}
	
	\begin{remark}
		The quantity $\eta$ is introduced for technical reasons concerning certain weight functions $\nu_D$ we employ when applying the transference principle. For further details, see the remarks preceding Lemma \ref{lem8.3}.
	\end{remark}
	
	The proof of Theorem \ref{thm2.7} is deferred to the final two sections of this paper. As in \cite{CC}, we prove this `linearised' result by applying an arithmetic regularity lemma. To streamline this forthcoming argument, we require the following proposition, which is a minor variation of \cite[Proposition 3.10]{CC} and is proved in the same way.
	
	\begin{prop}\label{prop2.9}
		Suppose that Theorem \ref{thm2.7} is true in the cases where $\gcd(L_1)=1$. Then,
		subject to altering the quantities $Z_0(D,r,\delta,L_1,L_2,P)$, $\eta$, and the implicit constant in the final bound, Theorem \ref{thm2.7} holds in general.
	\end{prop}

	\subsection{Sketch of the transference argument}
	
	In this subsection, we outline how the transference principle allows us to deduce Theorem \ref{thm2.5} from Theorem \ref{thm2.7}. Fix an integer polynomial $h$ that is intersective of the second kind, as well as a pair of linear forms $L_1$ and $L_2$ as in the statement of Theorem \ref{thm2.5}. We begin by recalling that Theorem \ref{thm2.5} concerns the equation
	\begin{equation}
		\label{eqn2.6}
		L_1(h(\bx)) = L_2(h(\by)),
	\end{equation}
	whilst Theorem \ref{thm2.7} considers, for some parameter $D\in\N$, the `linearised' equation
	\begin{equation}
		\label{eqn2.7}
		L_1(\bn) = L_2(h_D(\bz)).
	\end{equation}
	
	Suppose we have a finite colouring $\cP_X= \cC_1\cup\cdots\cup \cC_r$ and a set $A\subseteq\cP_X$ with $|A|\geqslant\delta|\cP_X|$. For the convenience of this sketch, assume that $X\equiv r_D \mmod{D}$. Choosing $Z\in\N$ such that $X=r_D + DZ$, we can define an $r$-colouring
	\[
	\{ z \in [Z]: r_D + Dz \in \cP \}
	= \Tilde{\cC}_1 \cup \cdots \cup \Tilde{\cC}_r
	\]
	by
	\begin{equation*}
		\Tilde{\cC_i}:=\{ z\in [Z] : r_D + Dz\in\cC_i\}.
	\end{equation*}
	Let $N:=h_D(Z)$. By pigeonholing, we find a `dense' set $\cA\subseteq[N]$ such that
	\begin{equation*}
		\cA\subseteq \left\{ \frac{h(x)-h(b)}{\lambda(D)}: x \in A \right \},
	\end{equation*}
	for some integer $b$. 
	
	Theorem \ref{thm2.7} now informs us that there are many solutions $(\bn,\bz)\in \cA^s\times\Tilde{\cC}_k^t$ to (\ref{eqn2.7}) for some $k\in[r]$. Given such a solution, our construction of $\cA$ and $\Tilde{\cC}_k$ furnishes a solution $(\bx,\by)\in A^s\times\cC_k^t$ to (\ref{eqn2.6}) satisfying
	\begin{equation*}
		n_i = \frac{h(x_i) - h(b)}{\lambda(D)}, \quad  y_j = r_D + Dz_j \qquad (1\leqslant i\leqslant s,
		\quad 
		1\leqslant j\leqslant t).
	\end{equation*}
	
	Since the map $(\bn,\bz)\mapsto(\bx,\by)$ is injective, this argument allows us to obtain many solutions to (\ref{eqn2.6}). However, observe that the number of solutions to (\ref{eqn2.6}) given in Theorem \ref{thm2.5} is of a different order to the number of solutions to (\ref{eqn2.7}) provided by Theorem \ref{thm2.7}. This is handled by instead considering weighted counts of solutions to (\ref{eqn2.6}). Our task is then to construct an appropriate weight function $\nu$ which is supported on the set
	\begin{equation*}
		[N]\cap\left\{ \frac{h(x)-h(b)}{\lambda(D)}: b < x \leqslant X\right \}.
	\end{equation*}
	The key utility of the transference principle is that, provided our weight function is suitably `pseudorandom', we can find a `dense model' $g:[N]\to[0,1]$ such that $\hat{\nu}\approx\hat{g}$. Applying Theorem \ref{thm2.7} to a set of the form $\cA=\{x\in[N]: g(x)>c\}$, our argument above allows us to prove Theorem \ref{thm2.5}.
	
	To ensure our weight $\nu$ is sufficiently pseudorandom, we have to contend with the fact that the set $h(\cP)$ is not equidistributed in residue classes. This issue prevents one from simply taking $\nu$ to be a scaled version of the indicator function of $h(\cP)$. Fortunately, there is a standard technical manoeuvre, known as the \emph{$W$-trick}, developed by Green \cite{Gre2005A} to account for equidistribution modulo small primes. In the setting discussed above, this amounts to demanding that our weight $\nu$ is supported on a set of the form
	\begin{equation*}
		[N]\cap\left\{ \frac{h(x)-h(b)}{\lambda(D)}: b < x \leqslant X, \quad x\equiv b \mmod{W\kappa} \right\},
	\end{equation*}
	for some $W,\kappa\in\N$ such that $W$ is divisible by all primes $p\leqslant w$ for some sufficiently large $w\in\N$. If we choose $D,W,\kappa$ appropriately, then we can ensure that the set
	\begin{equation*}
		\left\{ \frac{h(x)-h(b)}{\lambda(D)}: b < x \leqslant X, \quad x\equiv b \mmod{W\kappa} \right\}
	\end{equation*}
	equidistributes over congruence classes modulo $p$ for any prime $p\leqslant w$. The contribution of the remaining primes is then subsumed by the error term emerging from the transference of solutions from the `dense model' $g$ to $\nu$.
	The appearance of the additional parameter $\kappa\in\N$ here, resulting in a `double $W$-trick', was the main innovation of our previous work \cite{CC}. Its purpose is to ensure that $\lam(D)$ precisely accounts for all common divisors of the values of $h(x) - h(b)$, as $x$ ranges over the arithmetic progression $b$ modulo $W \kap$.
	
	\section{Linearisation and the \texorpdfstring{$W$}{W}-trick}
	\label{sec3}
	
	In this section, we execute the `double $W$-trick' and construct the weight function $\nu$ needed for our application of the transference principle. Throughout this section, we fix the parameters
	\begin{equation*}
		\delta,h,r,L_1,L_2
	\end{equation*}
	which appear in Theorem \ref{thm2.5}.
	
	\subsection{The \texorpdfstring{$W$}{W}-trick}
	
	Consider a set $A \subseteq \cP_X$ with $|A| \geqslant \del |\cP_X|$. Let $C \in \bN$ be large with respect to the fixed parameters, and let $w\in\N$ be large in terms of $C$. Define
	\[
	M = Cd^2 10^{4w},
	\qquad
	W = \left(
	\prod_{p \le w} p \right)
	^{100dw}, \qquad
	V = \sqrt W
	\]
	and
	\[
	D = W^2,
	\qquad
	Z = \frac{X - r_D}{D},
	\qquad
	N = h_D(Z) = \frac{h(X)}{\lam(D)}.
	\]
	Henceforth, we take $X\in\N$ sufficiently large in terms of $C,w,$ and the fixed parameters. We also assume that $D\mid(X-r_D)$, whence $Z \in \bN$.
	
	For $R \in \bN$, and $b \in [R]$, we write 
	\[
	A_{b,R} := \{ x \in A: x \equiv b \mmod R \}.
	\]
	We denote by $(H,W)_d$ the greatest $m \in \bN$ for which $m^d \mid (H,W)$. By \cite[Lemma A.5]{CC} and the Siegel--Walfisz theorem, we have
	\[
	\delta|\cP_X| \le |A| \le
	\sum_{\substack{
			b \in [W]: \\
			(b,W) = 1 \\
			(h'(b),W)_d \le M}} |A_{b,W}| + 
	O\left(10^w W M^{-1/2} 
	\frac{X}{\varphi(W) \log X}
	\right).
	\]
	Since $w$ is large relative to $C$ and $d$, we have $M<2^{50w}$. Hence, if $(h'(b),W)_d \le M$, then there cannot exist a prime $p\leqslant w$ which divides $(h'(b),W)$ with multiplicity greater than $50dw$. It follows that if $(h'(b),W)_d \le M$ then  $(h'(b),W)\mid V$. By incorporating the crude estimate
	\begin{equation*}
		\frac{W}{\varphi(W)} = \prod_{p\leqslant w}\left( 1- \frac{1}{p}\right)^{-1} \leqslant 2^w,
	\end{equation*}
	we find that
	\[
	\frac{\del X}{\log X}
	\ll \sum_{\substack{
			b \in [W]: \\
			(b,W) = 1 \\
			(h'(b),W) \mid V}} |A_{b,W}|.
	\]
	Thus, there exists $b_0 \in [W]$ such that
	\[
	|A_{b_0,W}| \gg 
	\frac{\del X}{\varphi(W) \log X},
	\qquad
	(b_0, W) = 1,
	\qquad
	(h'(b_0),W) \mid V.
	\]
	
	Define $\kap \in \bN$ by
	\[
	W \kap (h'(b_0),W) = \lam(D).
	\]
	Note that (\ref{eqn2.4}) implies that $\kappa$ is $w$-smooth, whence $\varphi(W)\kappa = \varphi(W\kappa)$. By pigeonholing, we can then find $b \in [W \kap]$ such that 
	\[
	b \equiv b_0 \mmod W,
	\qquad
	|A_{b,W\kap}| \gg
	\frac{\del X}{\varphi(W) \kap \log X}
	= \frac{\del X}{\varphi(W\kap) \log X}.
	\]
	Since $(h'(b),W)
	= (h'(b_0), W) \mid V$, we also have
	\[
	(h'(b),W\kap) = (h'(b),W).
	\]
	
	\subsection{The weight function}\label{subsec3.2}
	
	Our next task is to construct an appropriately `pseudorandom' weight function. Let $w,W$, and $\kappa$ be as defined in the previous subsection. Fix some $b\in[W\kappa]$ which satisfies
	\begin{equation}\label{eqn3.1}
		(h'(b),W)\mid V=\sqrt{W}.
	\end{equation}
	We then define
	\[
	\nu: \bZ \to [0,\infty),
	\qquad
	\nu(n) := 
	\frac{\varphi(W)}{W(h'(b),W)}
	\sum_{\substack{b < p \le X \\ p \equiv b \mmod{W \kap} \\ h(p) - h(b) = n \lam(D)}} h'(p) \log p.
	\]
	Observe that $\nu$ is supported on the set
	\begin{equation*}
		\left \{n\in\N: n=\frac{h(p) - h(b)}{\lam(D)},\; p\in\cP_X,\; p\equiv b \mmod{W\kappa} \right \}\subseteq [N].
	\end{equation*}
	
	Recall from the previous subsection that we are considering a fixed set $A\subseteq\cP_X$ with $|A|\geqslant\delta|\cP_X|$, and that we made a judicious choice of $b\in[W\kappa]$ so that $|A_{b,W\kappa}|$ is suitably dense. For this specific choice of $b$, let
	\begin{equation}
		\label{eqn3.2}
		\cA = \left \{
		\frac{h(p) - h(b)}{\lam(D)}: p \in A_{b,W\kap} \right \}.
	\end{equation}
	
	\begin{lemma} \label{lem3.1}
		Let $\cA,\nu$ be as defined above. If $X$ is sufficiently large in terms of $h,\delta,w$, then
		\[
		\sum_{n \in \cA} \nu(n) \gg_\del N.
		\]
	\end{lemma}
	
	\begin{proof} Let $c$ be a small, positive constant. For $X$ sufficiently large, the Siegel--Walfisz theorem implies that
		\[
		|A_{b,W\kap}| \ge 
		\# \{ p \le c\del X:
		\: 
		p \equiv b \mmod{W \kap} \},
		\]
		whence
		\[
		\sum_{p \in A_{b,W\kap}} h'(p) \log p \gg \sum_{\substack{p \le c\del X \\ p \equiv b \mmod{W \kap}}} p^{d-1} \log p.
		\]
		By the Siegel--Walfisz theorem again, we thus have
		\[
		\sum_{p \in A_{b,W\kap}} h'(p) \log p \gg_\del
		\frac{X^d}{\varphi(W \kap)}.
		\]
		Therefore
		\[
		\frac{W(h'(b),W)}{\varphi(W)} 
		\sum_{n \in \cA} \nu(n)
		= O((W\kap)^{d-1} \log W) +
		\sum_{p \in A_{b,W\kap}} h'(p) \log p \gg_\delta \frac{X^d}{\varphi(W \kap)}.
		\]
		Thus, for our choice of $\kappa$ and $b$, the desired bound now follows from the equalities
		\begin{equation*}
			\frac{W(h'(b),W)}{\varphi(W)} = \frac{\lambda(D)}{\kappa \varphi(W)} = \frac{\lambda(D)}{\varphi(W\kappa)}.
		\end{equation*}
	\end{proof}
	
	\section{Exponential sums}\label{sec4}
	
	In this section, we record some results on exponential sums of the form
	\begin{equation}\label{eqn4.1}
		\sum_{\substack{p \le t \\ p \equiv b \mmod m}}e(F(p))G(p),
	\end{equation}
	where $F$ is a real polynomial, and $G:(1,\infty)\to \R$ is a continuously differentiable function. We apply these results in \S\ref{sec5} to study the Fourier transform $\hat{\nu}$ of our weight function $\nu$. The results of this section are also used in \S\ref{sec9} to establish density bounds for `prime polynomial Bohr sets'.
	
	A standard observation in analytic number theory, going back over a century to Hardy and Littlewood, is that such exponential sums can only be large if their phases exhibit `major arc' behaviour. In the case of (\ref{eqn4.1}), this means that the leading coefficient of the polynomial $F$ must be very close to a rational number with small denominator. To elucidate this further, we record the following lemma from \cite{Hua1965}, which considers the situation where the leading coefficient of $F$ is rational. In what follows, and throughout this section, for all $k \in \bN$, let $\sig_k$ be large in terms of $k$ and put $C_k = 2^{8k} \sig_k$.
	
	\begin{lemma} \label{lem4.1} 
		Let $m \in \bN$ and $b \in \bZ$ be coprime. Let $F(y) \in \bR[y]$ have degree $k$, and suppose $a/q$ is its leading coefficient, where $a,q\in\bZ$ are coprime and
		\[
		(\log P)^{C_k} < q \le \frac{P^k}{(\log P)^{C_k}}.
		\]
		Assume that $P$ is sufficiently large in terms of $m$.
		Then
		\[
		\sum_{\substack{p \le P \\ p \equiv b \mmod m}} e(F(p)) \ll_k \frac{P}{(\log P)^{\sig_k+1}}.
		\]
	\end{lemma}
	
	\begin{proof}
		This follows immediately from \cite[Theorem 10]{Hua1965}.
	\end{proof}
	
	Using this lemma, we can show that (\ref{eqn4.1}) is small when the leading coefficient of $F$ is `minor arc', meaning that it is not well-approximated by a rational number with denominator at most polylogarithmic in $P$.
	
	\begin{lemma}
		\label{lem4.2}
		Let $m\in\bN$ and $b\in\bZ$ be coprime. Let $F(y)\in\bR[y]$ have degree $k$, and let $\theta$ be its leading coefficient. Let $G:(1,\infty)\to \R$ be a continuously differentiable function. Assume that $P$ is sufficiently large in terms of $m$, and that
		\[
		\max \{ q,
		P^k \| q \tet \| \}
		> (\log P)^{2 C_k}
		\qquad (q \in \bN).
		\]
		Then
		\[
		\sum_{\substack{p \le P \\ p \equiv b \mmod m}} e(F(p)) G(p)\log p \ll_k \frac{P}{(\log P)^{\sigma_k}}\cdot\lVert G\rVert_{\cS[2,P]}.
		\]
	\end{lemma}
	
	\begin{proof}
		By Dirichlet's approximation theorem, there exist coprime $q \in \bN$ and $a \in \bZ$ such that 
		\[
		q \le \frac{P^k}{(\log P)^{2 C_k}},
		\qquad
		|q \tet - a| \le 
		\frac{(\log P)^{2 C_k}}{P^k}.
		\]
		By our assumption, we also have 
		\[
		q > (\log P)^{2 C_k}.
		\]
		Thus $\bet := \tet - a/q$ satisfies
		\[
		|\bet| \le P^{-k}.
		\]
		
		Let $f(y) = F(y) - \bet y^k$.  By partial summation \cite[Lemma 2.6]{Vau1997}, we have
		\begin{equation*}
			\sum_{\substack{p \le P \\ p \equiv b\mmod m}} e(F(p)) G(p)\log p = A(P)\psi(P) - \int_{2}^{P}A(t)\psi'(t) \d t,
		\end{equation*}
		where
		\[
		\psi(t) := e(\beta t^k)G(t)\log t, \qquad
		A(t) := \sum_{\substack{p \le t \\ p \equiv b \mmod m}}
		e(f(p)).
		\]
		We deduce from Lemma \ref{lem4.1} and the trivial bound $|A(t)| \le t$ that
		\begin{equation}
			\label{eqn4.2}
			A(t) \ll \frac{P}{(\log P)^{\sigma_k +1}} \qquad \left( 2 \le t \le P \right).
		\end{equation}
		This implies that
		\begin{equation*}
			A(P)\psi(P)(\log P)^{\sigma_k} \ll_k PG(P). 
		\end{equation*}
		It therefore remains to estimate 
		\begin{equation*}
			\int_{2}^{P}A(t)\psi'(t) \d t = I_1 + I_2 + I_3 + I_4,
		\end{equation*}
		where
		\begin{align*}
			I_1 = \int_{2}^{P}A(t)G'(t)e(\beta t^k)\log t \d t, \qquad
			& I_2 = \beta k\int_{2}^{P}t^{k-1}A(t)G(t)e(\beta t^k)\log t \d t,\\
			I_3 = \int_{2}^{P/(\log P)^{\sigma_k}}(A(t)/t)G(t)e(\beta t^k) \d t, \qquad
			&I_4 = \int_{P/(\log P)^{\sigma_k}}^{P}(A(t)/t)G(t)e(\beta t^k) \d t.
		\end{align*}
		The bound (\ref{eqn4.2}) gives
		\begin{equation*}
			I_1(\log P)^{\sigma_k} \ll P^2 \max_{2\leqslant t\leqslant P}|G'(t)|. 
		\end{equation*}
		Similarly, since $|\beta|\leqslant P^{-k}$, we see that
		\begin{equation*}
			I_2 (\log P)^{\sigma_k} \ll P \max_{2\leqslant t\leqslant P}|G(t)|.
		\end{equation*}
		Using the trivial bound $|A(t)|\leqslant t$, we have
		\begin{equation*}
			I_3(\log P)^{\sigma_k} \ll P\max_{2\leqslant t\leqslant P/(\log P)^{\sigma_k}}|G(t)|.
		\end{equation*}
		Similarly, we deduce from (\ref{eqn4.2}) that
		\begin{equation*}
			I_4(\log P)^{\sigma_k} \ll \int_{P/(\log P)^{\sigma_k}}^P |G(t)| \d t \leqslant P\max_{2\leqslant t\leqslant P}|G(t)|.
		\end{equation*}
		Combining these estimates completes the proof.
	\end{proof}
	
	The above two lemmas suffice to handle `minor arc' behaviour. As is typical in applications of the circle method, we treat the major arcs by establishing asymptotic formulae for the exponential sums (\ref{eqn4.1}).
	
	\begin{lemma}
		[General major arc asymptotic]
		\label{lem4.3}
		Let $f(y)\in\bZ[y]$ have degree $k$, and let $G:(1,\infty)\to \R$ be a continuously differentiable function.
		Let $b \in \bZ$ and $m \in \bN$ be coprime, and let $Q \in \bN$ with
		\begin{equation} 
			\label{eqn4.3}
			\frac{f(b+mx) - f(b)}{Q} \in \bZ[x].
		\end{equation}
		Let $\tet \in \bR$ and $P \ge 2$, and suppose 
		$(q,a) \in \bN \times \bZ$ with
		$(a,q) = 1$ and 
		\[
		q \ll (\log P)^{2 C_k}, 
		\qquad
		|q \theta - a| \ll \frac{Q (\log P)^{2C_k}}{P^k}.
		\]
		Let $c > 0$ be constant, small in terms of $C_k$, and put $\bet = \theta - a/q$. If $P$ is sufficiently large relative to $m$ and $Q$, then
		\[
		\sum_{\substack{p \le P \\ p \equiv b \mmod m}} e_Q (
		\tet f(p)) G(p)\log p = I_{f,G}(\bet) \frac{S(q,a;m)}{\varphi(mq)} + 
		O_f(Pe^{-c \sqrt{\log P}}
		\lVert G\rVert_{\cS[2,P]}),
		\]
		where
		\[
		I_{f,G}(\bet) = \int_2^P e_Q
		(\bet f(t)) G(t) \d t,
		\qquad
		S(q,a; m) = \sum_{\substack{t \mmod{mq} \\ (t,q) = 1 \\ t \equiv b \mmod m}} e_{Qq} (af(t)).
		\]
	\end{lemma}
	
	\begin{remark} The condition \eqref{eqn4.3} holds if $f(b+mx)/Q \in \bZ[x]$.
	\end{remark}
	
	\begin{proof}
		Writing $g(x)\in\bZ[x]$ for the integer polynomial appearing in \eqref{eqn4.3}, if $u \in \bZ$ then
		\[
		\frac{f(b+m(u + qv)) - f(b)}{Qq} - \frac{g(u)}{q} = \frac{g(u+qv) - g(u)}{q} \in\bZ[v].
		\]
		This implies that
		\[
		e_{Qq}(f(b+m(u + qv))) = e_{Qq}(f(b + mu)) \qquad (v\in\bZ).
		\]
		Hence, for $n \le P$,
		\[
		S_n := \sum_{\substack{p \le n \\ p \equiv b \mmod m}} e_{Qq}(
		af(p)) 
		= O(mq) +
		\sum_{\substack{t \mmod{mq} \\ (t,q) = 1 \\ t \equiv b \mmod m}} e_{Qq} \left(a f(t)\right)
		\sum_{\substack{p \le n \\ 
				p \equiv t \mmod{mq}}} 1.
		\]
		By the Siegel--Walfisz theorem (Theorem \ref{thm2.1}), the inner sum is
		\[
		\sum_{\substack{p \le n \\ 
				p \equiv t \mmod{mq}}} 1
		= \frac{\Li(n)}{\varphi(mq)}
		+ O(Pe^{-3c \sqrt{\log P}}),
		\]
		whence
		\[
		S_n = \frac{\Li(n)}{\varphi(mq)} S(q,a;m) + O(P e^{-2c \sqrt{\log P}}).
		\]
		Writing $\psi(t) = e_Q(\bet f(t)) G(t)\log t$, summation by parts gives
		\begin{align*}
			\sum_{\substack{p \le P \\ p \equiv b \mmod m}} e_Q\left(
			F(p) \right) G(p) \log p &= \sum_{n \le P} (S_n - S_{n-1}) \psi(n)
			\\ &= S_P \psi(P+1) + \sum_{n \le P} S_n (\psi(n) - \psi(n+1)).
		\end{align*}
		
		By hypothesis, for $P$ sufficiently large,
		\begin{equation*}
			|\beta| =
			|\theta - a/q| \ll \frac{Q(\log P)^{2C_k}}{qP^k} \ll \frac{(\log P)^{2C_k+1}}{P^k}.
		\end{equation*}
		Hence, for all $x,y\in[2,P]$ with $x<y$, the mean value theorem yields
		\begin{align}
			\notag
			\left\lvert\frac{\psi(y) - \psi(x)}{y-x}\right\rvert
			&\leqslant \sup_{t\in[x,y]}\left\lbrace |G(t)/t| + |G'(t)\log t|+ |\beta f'(t)G(t)\log t| \right\rbrace\\
			\label{eqn4.4}
			&\ll_f \lVert G\rVert_{L^{\infty}[2,P]}\left(\frac{1}{x} + \frac{(\log P)^{2(C_k + 1)}}{P}\right) + \lVert G'\rVert_{L^{\infty}[2,P]}\log P.
		\end{align}
		In particular, this shows that
		\begin{align*}
			\sum_{n \le P} |\psi(n) - \psi(n+1)| &\ll \lVert G\rVert_{L^{\infty}[2,P]}\left((\log P)^{2(C_k + 1)}+ \sum_{n\leqslant P} \frac{1}{n}\right) + \lVert G'\rVert_{L^{\infty}[2,P]}P\log P
			\\ &\ll (\log P)^{2(C_k+1)}\lVert G\rVert_{\cS[2,P]}.
		\end{align*}
		As $\Li(t) = \sum_{n=3}^t \int_{n-1}^n \frac{\d x}{\log x}$, summation by parts now gives
		\begin{align*}
			&\sum_{\substack{p \le P \\ p \equiv b \mmod m}} e_Q\left(
			F(p) \right) G(p) \log p + O(P e^{-c \sqrt{\log P}}\lVert G\rVert_{\cS[2,P]}) \\
			&= \frac{S(q,a;m)}{\varphi(mq)}
			\left( \Li(P)\psi(P+1) + 
			\sum_{n \le P} \Li(n) (\psi(n) - \psi(n+1))
			\right)
			\\
			&= \frac{S(q,a;m)}{\varphi(mq)} 
			\sum_{3 \le n \le P}
			\int_{n-1}^n \frac{\psi(n)}{\log x} \d x.
		\end{align*}
		Note that
		\[
		\sum_{3 \le n \le P}
		\int_{n-1}^{n}
		\frac{\d x}
		{(n-1)\log x}
		\ll \sum_{3 \le n \le P}
		\int_{n-1}^{n}
		\frac{\d x}
		{x\log x} = \log \log P - \log \log 2.
		\]
		% \begin{equation*}
			% \sum_{3 \le n \le P}
			% \int_{n-1}^{n}
			% \frac{\d x}
			% {(n-1)\log x}
			% = \sum_{3 \le n \le P} \frac{\Li(n)}{n(n-1)} \ll \log P.
			% \end{equation*}
		Thus, using \eqref{eqn4.4} to replace each $\psi(n)$ by $\psi(x)$, we obtain
		\begin{align*}
			\sum_{3 \le n \le P}
			\int_{n-1}^n \frac{\psi(n)}{\log x} \d x
			&= I_{f,G}(\bet) + O((\log P)^{2(C_k+1)}\lVert G\rVert_{\cS[2,P]}),
		\end{align*}
		which completes the proof.
	\end{proof}
	
	\section{Fourier decay}\label{sec5}
	
	Returning to the study of our weight function $\nu$, we need to show that it is suitably `pseudorandom'. This will then allow us to `transfer' solutions from the linearised equation (\ref{eqn2.7}) to our original equation (\ref{eqn2.6}). As in \cite{CC,Chow2018,CLP2021,Pre2021}, we accomplish this via a \emph{Fourier decay} estimate (together with the restriction estimates from the next section).
	
	\begin{lemma}
		\label{lem5.1}
		Let $\nu$ be as defined above, where $b\in[W\kappa]$ satisfies (\ref{eqn3.1}), and assume that $X$ is sufficiently large in terms of $w$. Then, for all $\alpha\in\T$, 
		\begin{equation} \label{eqn5.1}
			|\hat \nu(\alp) -
			\widehat{1_{[N]}} (\alp)|
			\ll_{h,\eps} w^{\eps-1/d} N.
		\end{equation}
	\end{lemma}
	
	\begin{remark}
		As in \cite[\S5]{CC}, the above lemma does not rely upon nor make any reference to sets $A\subseteq\cP_X$ or $\cA\subseteq[N]$.
	\end{remark}
	
	We study the Fourier transform $\hat{\nu}$ using the Hardy--Littlewood circle method and the exponential sum estimates established previously. We define the set of \emph{minor arcs}
	\begin{equation*}
		\fm :=\left\lbrace\alpha\in\T : \max \{q, X^d \| q \alp \| \} > (\log X)^{2C_d} \text{ for all } q\in\bN\right\rbrace.
	\end{equation*}
	The set of \emph{major arcs} $\fM:=\T\setminus\fm$ therefore consists of all $\alpha\in\T$ for which there exist $a,q\in\Z$ such that
	\begin{equation}
		\label{eqn5.2}
		1 \le q \le (\log X)^{2C_d},
		\qquad
		(q,a) = 1,
		\qquad
		X^d |q \alp - a| \le 
		(\log X)^{2C_d}.
	\end{equation}
	
	For convenience, we recall that
	\[
	\hat \nu(\alp) =
	\frac{\varphi(W)}{W(h'(b),W)}
	\sum_{\substack{
			b < p \le X \\
			p \equiv b \mmod{W \kap}}}
	h'(p) \log p \cdot e \left(
	\alp \frac{h(p)-h(b)}
	{\lam(D)} \right) \quad (\alpha\in\bT).
	\]
	We therefore observe that the results of the previous section may be applied to estimate $\hat \nu(\alp)$ upon taking
	\begin{equation}
		\label{eqn5.3}
		\tet = \alp, \quad f(y) = h(y) - h(b), \quad G=h', \quad Q=\lam(D),\quad m=W\kappa,
		\quad P = X.
	\end{equation}
	With this choice of parameters, we compute that
	\begin{equation}
		\label{eqn5.4}
		\lVert G\rVert_{\cS[2,P]} \ll_h X^{d-1}
	\end{equation}
	and
	\begin{equation}
		\label{eqn5.5}
		I_{f,G}(\beta)
		= \int_2^X e(\bet h(x)/\lam(D)) h'(x) \d x
		= \lam(D)\left( O_h(1) + \int_{0}^{N} e(\bet y)\d y\right).
	\end{equation}
	
	Our proof of Lemma \ref{lem5.1} for $\alpha\in\fm$ proceeds by the same strategy as in \cite[\S4]{Chow2018}: we show that $\hat{\nu}(\alpha)$ and $\hat{1}_{[N]}(\alpha)$ are both far smaller than the required upper bound. This is encapsulated in the following corollary of Lemma \ref{lem4.2}.
	
	\begin{cor}
		[Minor arc estimate]
		\label{cor5.3}
		If $\alpha\in\fm$, then
		\begin{equation*}
			\hat{\nu}(\alpha) \ll_{h} X^d(\log X)^{-\sigma_d} \qquad \text{and} \qquad \widehat{1_{[N]}} (\alp) \ll X^d(\log X)^{-2C_d}.
		\end{equation*}
	\end{cor}
	
	\begin{proof}
		In view of (\ref{eqn5.3}) and (\ref{eqn5.4}), the first estimate follows immediately from Lemma \ref{lem4.2}. For the second estimate, we deduce from the definition of $\fm$ that
		\begin{equation*}
			\widehat{1_{[N]}} (\alp) = \sum_{n=1}^{N}e(\alpha n) \ll \lVert\alpha\rVert^{-1} \leqslant X^d(\log X)^{-2C_d} \qquad (\alpha\in\fm),
		\end{equation*}
		as required.
	\end{proof}
	
	We similarly establish an asymptotic formula for $\hat{\nu}(\alpha)$ on the major arcs by invoking Lemma \ref{lem4.3}. Define
	\[
	S(q,a) := \sum_{\substack{t \mmod{W \kap q} \\ (t,q) = 1 \\ t \equiv b \mmod{W \kap}}} e_q \left(
	a \frac{h(t) - h(b)}{\lam(D)} \right), \qquad I(\bet) := 
	\int_0^N e(\bet y) \d y.
	\]
	
	\begin{cor}
		[Major arc asymptotic] 
		\label{cor5.4}
		Suppose $(\alp,q,a) \in \bR \times \bN \times \bZ$ with \eqref{eqn5.2},
		and put $\bet = \alp - a/q$. Let $c > 0$ be constant, small in terms of $C_d$. Then
		\[
		\hat \nu(\alp) = \frac{\varphi(W\kap)}{\varphi(W\kap q)}
		S(q,a) I(\bet)
		+ O_h(Ne^{-c \sqrt{\log X}}).
		\]
	\end{cor}
	
	\begin{proof}
		Recall that $\lambda(D)N = h(X)\asymp_h X^d$.
		Thus, by combining (\ref{eqn5.3}) with (\ref{eqn5.4}) and (\ref{eqn5.5}), the desired formula is provided by Lemma \ref{lem4.3}.
	\end{proof}
	
	To elucidate this formula further, we estimate $S(q,a)$.
	
	\begin{lemma}
		\label{lem5.5}
		Let $a,q\in\Z$ with $q \ge 2$ and $(q,a) = 1$. Then
		\[
		S(q,a) \ll_{h,\eps} \min\left\lbrace q^{\eps - 1/d}\varphi(q),\; \frac{\varphi(W\kap q)}
		{\varphi(W \kap)}w^{\eps-1/d}\right\rbrace.
		\]
		Furthermore, if $(q,W) > 1$, then $S(q,a) = 0$.
	\end{lemma}
	
	\begin{proof}
		Write $q = q_1q_2$, where $q_1 \in \bN$ is $w$-smooth and $q_2$ is $w$-rough. Observe that
		\[
		S(q,a) = \sum_{\substack{
				x \mmod q \\ 
				(W\kap x + b, q) = 1}} 
		e_q(ag(x)),
		\]
		where
		\[
		g(x) = \frac{h(W \kap x + b) - h(b)}{\lam(D)} =
		\frac{h(W \kap x + b) - h(b)}{W\kap (h'(b), W\kap)}
		\in \bZ[x]
		\]
		by Taylor's theorem. As $(q_1,q_2) 
		= 1$, a standard calculation reveals that 
		\[
		S(q,a) = S(q_1,A_1) S(q_2,A_2), 
		\]
		where
		\[
		\frac{a}{q} = \frac{A_1}{q_1} + \frac{A_2}{q_2},
		\]
		as is noted in the proof of \cite[Lemma 9]{Rice2013}. Observe that $(A_1,q_1) = (A_2,q_2) = 1$.
		
		As $q_1$ is $w$-smooth and $(b,W)=1$, we always have $(W\kap x+b, q_1) = 1$.
		Let
		\[
		H = (q_1,W), \qquad
		q_1 = Hq_1', \qquad
		W = HW',
		\]
		so that $(q_1',W') = 1$. Writing $x = y + q_1'z$ and 
		\[
		g(x) = v_d x^d + \cdots 
		+ v_1 x
		\]
		gives
		\begin{align*}
			S(q_1,A_1) &= 
			\sum_{y \mmod{q_1'}} \:
			\sum_{z \mmod H} 
			e_{Hq_1'}(A_1 \sum_{j \le d}
			v_j (y + q_1'z)^j).
		\end{align*}
		As
		\[
		(h'(b),W) \mid V = \sqrt W,
		\qquad
		W \mid \kap (h'(b),W),
		\]
		we must have $(h'(b),W) \mid V \mid \kap$. Thus, for $2 \le j \le d$, we have
		\[
		v_j = \frac{h^{(j)}(b) 
			(W \kap)^{j-1}}
		{j! (h'(b),W)} \equiv 0 \mmod W.
		\]
		Now
		\[
		S(q_1,A_1) = \sum_{y \le q'_1} e_{q_1} (A_1 g(y))
		\sum_{z \le H} 
		e_H(A_1 v_1 z),
		\qquad
		v_1 = \frac{h'(b)}{(h'(b),W)}.
		\]
		For each prime $p \le w$, we have $\ord_p(h'(b)) < \ord_p(W)$, and so $\ord_p(v_1) = 0$. Therefore $(v_1,W) = 1$, so $(H, A_1 v_1 ) =1$, whence
		\[
		S(q_1,A_1) = \begin{cases}
			1, &\text{if } q_1 = 1 \\
			0, &\text{if } q_1 \ge 2.
		\end{cases}
		\]
		This completes the proof of the assertion that $S(q,a)=0$ whenever $(q,W)>1$.
		
		In view of this result, we may henceforth assume that $q_1=1$ and $q_2 = q \ge 2$. In particular, we have $q > w$. Let us denote by $\ell_h$ the leading coefficient of $h$. Then
		\[
		v_d = \frac{\ell_h (W \kap)^{d-1}} {(h'(b),W)},
		\qquad (q,W) = 1,
		\]
		so $(v_d, q) \ll_h 1$. Now Lemma \ref{lem2.2} provides a constant $C=C(d) > 1$ such that
		\[
		S(q,a) = S(q, A_2) \ll_h
		C^{\ome(q)} q^{1-1/d}
		\ll_{d,\eps} q^{\eps - 1/d} \varphi(q)< \frac{\varphi(W\kap q)}{\varphi(W \kap)} w^{\eps-1/d},
		\]
		as required.
	\end{proof}
	
	These two results allow us to establish our Fourier decay estimate.
	
	\begin{proof}
		[Proof of Lemma \ref{lem5.1}]
		Corollary \ref{cor5.3} gives \eqref{eqn5.1} for all $\alpha\in\fm$. We henceforth assume that $\alpha\in\fM$.
		We begin by considering the case where \eqref{eqn5.2} holds with $q =1$. As demonstrated in \cite[\S 4]{Chow2018}, Euler–Maclaurin summation delivers the bound
		\begin{equation*}
			\widehat{1_{[N]}} (\alp)
			- I(\alpha) \ll_h 
			(\log X)^{2C_d}.
		\end{equation*}
		Applying the triangle inequality and Corollary \ref{cor5.4} therefore gives
		\begin{equation*}
			|\hat{\nu}(\alpha) - \widehat{1_{[N]}} (\alp)| \leqslant |\hat{\nu}(\alpha) - I(\alpha)| + |\widehat{1_{[N]}} (\alp) -I(\alpha)| \ll_h Ne^{-c\sqrt{\log X}}.
		\end{equation*}
		
		Finally, suppose that \eqref{eqn5.2} holds with $q \ge 2$. We note from \cite[Equation (4.3)]{Chow2018} that
		\[
		\widehat{1_{[N]}} (\alp)
		\ll q \le (\log X)^{2 C_d}.
		\]
		Thus, in view of the trivial estimate $|I(\bet)| \le N$, the desired result follows by combining Corollary \ref{cor5.4} with Lemma \ref{lem5.5}.
	\end{proof}
	
	Before moving on, we record the following consequence of Corollary \ref{cor5.4} and Lemma \ref{lem5.5}, which we use in the next section.
	
	\begin{cor}[Major arc estimate]
		\label{cor5.6}
		Let $\alpha\in\T$ and $q\in\N$. If \eqref{eqn5.2} holds for some $a \in \bZ$, then
		\begin{equation*}
			\hat \nu(\alp)
			\ll_{h,\eps} q^{\eps - 1/d}
			\min \{ N, \| \alp - a/q \|^{-1} \} + 
			O(Ne^{-c \sqrt{\log X}}).
		\end{equation*}
	\end{cor}
	
	\begin{proof}
		Integrating by parts delivers the standard estimate
		\begin{equation}
			\label{eqn5.6}
			I(\bet) \ll \min \{N,
			\| \bet \|^{-1} \} \qquad (\beta\in\T).
		\end{equation}
		Incorporating the elementary inequality $\varphi(W\kappa)\varphi(q)\leqslant\varphi(W\kappa q)$ and Lemma \ref{lem5.5} delivers
		\begin{equation*}
			\frac{\varphi(W\kappa)}{\varphi(W\kappa q)}S(q,a)I(\alp - a/q) \ll_{h,\eps} q^{\eps-1/d}
			\min \{N,\|  \alp - a/q \|^{-1} \}.
		\end{equation*}
		The required result now follows from Corollary \ref{cor5.4}.
	\end{proof}
	
	\section{Restriction estimates}\label{sec6}
	
	Recall from \S\ref{sec3} that $\nu$ is supported on the set
	\begin{equation*}
		\left \{n\in[N]: n=\frac{h(p) - h(b)}{\lam(D)},\; p\in\cP_X,\; p\equiv b \mmod{W\kappa} \right \}.
	\end{equation*}
	After linearising, we wish to solve (\ref{eqn2.7}) with the $n_i$ drawn from a dense subset $\cA$ of the above set. This leads us to the study of functions $\phi:\bZ\to\bC$, such as the indicator function of $1_{\cA}$, which are majorised by $\nu$, meaning that $|\phi|\leqslant\nu$.
	
	The purpose of this section is to establish two restriction estimates. These will then be used in the next section to execute the transference argument. The first restriction estimate is for the weight $\nu$ and is needed to transfer between `dense variables' $x_i$ and $n_i$ in Equations (\ref{eqn2.6}) and (\ref{eqn2.7}), respectively. Our second restriction estimate concerns an auxiliary weight function $\nu_D$ and is required for interpolation.
	
	Throughout this section, we define $\nu$ as in \S\ref{subsec3.2} for a fixed choice of $b\in[W\kappa]$ satisfying (\ref{eqn3.1}). We also let $T = T(d)$ be as in \S\ref{sec2}.
	
	\subsection{Restriction I}
	
	We begin with the following restriction estimate for $\nu$. 
	
	\begin{prop} \label{prop6.1} Let $E > 2T$, and let $\phi: \bZ \to \bC$ with $|\phi| \le \nu$. Then
		\[
		\int_\bT |\hat \phi(\alp)|^E \d \alp \ll_{h,E} N^{E-1}.
		\]
	\end{prop}
	
	This is easily bootstrapped to the following restriction estimate for $\nu + 1_{[N]}$, see the deduction of \cite[Lemma 6.2]{CC}.
	
	\begin{prop} \label{prop6.2} 
		Let $E > 2T$, and let $\phi: \bZ \to \bC$ with $|\phi| \le \nu + 1_{[N]}$. Then
		\[
		\int_\bT |\hat \phi(\alp)|^E \d \alp \ll_{h,E} N^{E-1}.
		\]
	\end{prop}
	
	To prove Proposition \ref{prop6.1}, we proceed in stages. We introduce the auxiliary function
	\[
	\mu: \bZ \to [0,\infty),
	\qquad
	\mu(n) = \frac{1}{(h'(b),W)}
	\sum_{\substack{
			b < x \le X \\
			x \equiv b \mmod{W \kap} \\
			h(x) - h(b) = n \lam(D)}}
	h'(x),
	\]
	noting that $\nu \le (\log X) \mu$ pointwise. We compute that
	\begin{align*}
		\| \mu \|_1 &= (h'(b),W)^{-1}
		\sum_{\substack{b < x \le X \\
				x \equiv b \mmod{W \kap}}} h'(x) \\
		&\ll (h'(b),W)^{-1}
		\sum_{y \le X/W \kap} 
		(W \kap y)^{d-1} \\
		&\ll \frac{X^d}{W \kap (h'(b),W)}
		= \frac{X^d}{\lam(D)}
		\ll N.
	\end{align*}
	
	We begin with an epsilon-slack restriction estimate for $\mu$.
	
	\begin{lemma} 
		[Epsilon-slack estimate]
		\label{lem6.3}
		Let $\psi: \bZ \to \bC$ with $|\psi| \le \mu$. Then
		\[
		\int_\bT |\hat \psi(\alp)|^{2T} \d \alp \ll_{h,\eps} N^{2T-1+\eps}.
		\]
	\end{lemma}
	
	\begin{proof} Observe that 
		$\| \psi \|_\infty \le 
		\| \mu \|_\infty \ll 
		X^{d-1}$. By orthogonality and the triangle inequality,
		\begin{align*}
			\int_\bT |\hat \psi(\alp)|^{2T} \d \alp
			&= \sum_{n_1 + \cdots + n_T = n_{T+1} + \cdots + n_{2T}} \psi(n_1) \cdots \psi(n_T) \overline{
				\psi(n_{T+1}) \cdots 
				\psi(n_{2T})} \\
			&\ll (X^{d-1})^{2T} 
			\# \{ \bx \in [X]^{2T}: h(x_1) + \cdots + h(x_T) = h(x_{T+1}) + \cdots + h(x_{2T}) \} \\ 
			&\ll X^{2T(d-1) + 2T - d + \eps} \ll N^{2T-1+2\eps}.
		\end{align*}
	\end{proof}
	
	Our next goal is to largely remove $\eps$ from the exponent in this restriction estimate for $\mu$, obtaining a log-slack estimate by passing to a slightly higher moment. To accomplish this, we require some bounds on
	\[
	\hat \mu(\tet) = 
	\frac1{(h'(b),W)}
	\sum_{\substack{
			b < x \le X \\
			x \equiv b \mmod{W \kap}}} 
	h'(x) 
	e\left(\tet \frac{h(x)-h(b)}{\lam(D)} \right).
	\]
	The triangle inequality and partial summation yield
	\begin{equation*}
		\label{TrianglePartial}
		\hat \mu(\tet)
		\ll \frac{X^{d-1}}{(h'(b),W)} \left( 1 +
		\max_{X^{1/2} \le P \le X}
		|g(\tet;P)| \right),
	\end{equation*}
	where
	\[
	g(\tet;P) = \frac1{(h'(b),W)}
	\sum_{\substack{
			b < x \le X + b \\
			x \equiv b \mmod{W \kap}}}
	e\left(\tet \frac{h(x)-h(b)}{\lam(D)} \right).
	\]
	Writing $x = W\kap y + b$ gives
	\[
	\sum_{\substack{
			x \le X \\
			x \equiv b \mmod{W \kap}}} e\left(\tet \frac{h(x)-h(b)}{\lam(D)} \right)
	= \sum_{y \le X/(W\kap)}
	e\left(
	\frac{c_d (W\kap)^d}{\lam(D)} \tet f_d(y)
	\right),
	\]
	where $c_d$ is the leading coefficient of $h$ and $f_d(y) \in \bZ[y]$ is monic of degree $d$.
	
	Suppose $X^{1/2} \le P \le X$.
	The exponential sum 
	\[
	g_1(\alp; P) := 
	\sum_{y \le P/(W\kap)}
	e(\alp f_d(y))
	\]
	can be treated using Roger Baker's estimates \cite{Bak1986}. We apply the formulation \cite[Lemma 2.3]{Chow2016}, noting from its proof that the quantity $\sig(d)$ therein can be replaced by $2^{1-d}$. This delivers the following conclusion.
	
	\begin{lemma} \label{lem6.4}
		If $|g_1(\alp; P)| > (P/(W\kap))^{1 - 2^{1-d} + \eps}$ then there exist coprime $r \in \bN$ and $b \in \bZ$ such that
		\[
		g_1(\alp; P) \ll_{d,\eps} r^{\eps-1/d}
		P(W\kap)^{-1}
		(1 + (P/(W\kap))^d 
		|\alp - b/r|)^{-1/d}.
		\]
	\end{lemma}
	
	\begin{lemma} 
		\label{lem6.5}
		Let 
		\[
		\fn = \{ \tet \in \bT:
		|\hat \mu(\tet)| 
		\le X^{d-2^{-d}} \}.
		\]
		If $\tet \in \bT \setminus \fn$, then there exist coprime $q \in \bN$ and $a \in \bZ$ such that
		\[
		\hat \mu(\tet) \ll_{d,\eps}
		N (\log X) q^{\eps-1/d} 
		(1 + N|\tet - a/q|)^{-1/d}.
		\]
	\end{lemma}
	
	\begin{proof} Let
		$P \in [X^{1/2}, X]$ maximise
		\[
		\left|g_1\left(
		\frac{c_d (W\kap)^d}{\lam(D)} \tet; P \right)\right|.
		\]
		By \eqref{TrianglePartial},
		\[
		\left|g_1\left(
		\frac{c_d (W\kap)^d}{\lam(D)} \tet; P \right)\right| > (P/(W\kap))^{1 - 2^{1-d} + \eps}.
		\]
		By Lemma \ref{lem6.4}, there exist coprime $r \in \bN$ and $b \in \bZ$ such that
		\begin{align*}
			g_1\left(
			\frac{c_d (W\kap)^d}{\lam(D)} \tet; P \right)
			&\ll r^{\eps-1/d}
			P(W\kap)^{-1}
			\left(1 + (P/(W\kap))^d 
			\left|\frac{c_d (W\kap)^d}{\lam(D)} \tet - b/r \right| \right)^{-1/d} \\
			&\ll r^{\eps-1/d} 
			X(W\kap)^{-1}
			\left(1 + (X/(W\kap))^d 
			\left|\frac{c_d (W\kap)^d}{\lam(D)} \tet - b/r \right| \right)^{-1/d}.
		\end{align*}
		With
		\[
		a = \frac{b}
		{(b, c_d (W\kap)^d/\lam(D))},
		\qquad
		q = \frac{rc_d (W\kap)^d/\lam(D)}
		{(b, c_d (W\kap)^d/\lam(D))},
		\]
		we now have
		\begin{align*}
			g_1\left(
			\frac{c_d (W\kap)^d}{\lam(D)} \tet; P \right)
			&\ll r^{\eps-1/d} 
			\frac{X}{W\kap}
			\left(1 + \frac{X^d}{\lam(D)} |\tet-a/q| \right)^{-1/d} \\
			&\ll q^{\eps-1/d} 
			\frac{X\log X}{W\kap}
			\left(1 + \frac{X^d}{\lam(D)} |\tet-a/q| \right)^{-1/d},
		\end{align*}
		since $X$ is large in terms of $w$. Finally, recall that
		\[
		N \asymp \frac{X^d}{\lam(D)}
		= \frac{X^d}{W\kap(h'(b),W)}
		\]
		and
		\[
		X^{d-2^{-d}} <
		|\hat{\mu}(\tet)| \ll
		X^{d-1} + \frac{X^{d-1}}
		{(h'(b),W)} g_1\left(
		\frac{c_d (W\kap)^d}{\lam(D)} \tet; P \right).
		\]
	\end{proof}
	
	By passing to a slightly higher moment, we are now able to obtain a log-slack analogue of Lemma \ref{lem6.3}.
	
	\begin{lemma} 
		[Log-slack estimate]
		\label{lem6.6}
		Let $v > 2T$ be real, and let $\psi: \bZ \to \bC$ with $|\psi| \le \mu$. Then
		\[
		\int_\bT |\hat \psi(\alp)|^v \d \alp \ll_{h,v} N^{v-1} (\log X)^v.
		\]
	\end{lemma}
	
	\begin{proof} 
		Observe from its proof that the general epsilon-removal lemma \cite[Lemma 25]{Sal2020} holds with
		$\| \alp - a/q \|^\kap$ in place of $\| \alp - a/q \|$, with the notation therein.
		Inserting Lemmas \ref{lem6.3} 
		and \ref{lem6.5} into this gives
		\[
		\int_\bT 
		(|\hat \psi(\alp)|
		/\log X)^v 
		\d \alp \ll N^{v-1}.
		\]
	\end{proof}
	
	\begin{proof} [Proof of Proposition \ref{prop6.1}]
		Let $v \in (2T, E)$. Applying Lemma \ref{lem6.6} to 
		$\psi = (\log X)^{-1} \phi$ gives the log-slack restriction estimate
		\[
		\int_\bT |\hat \phi(\alp)|^v \d \alp \ll_{h,v}
		N^{v-1} (\log X)^{2v}.
		\]
		We can choose $\sig_d$ to be large in terms of $E$ and $v$. Thus, in view of Corollaries \ref{cor5.3} and \ref{cor5.6}, the desired result follows from the general epsilon-removal lemma \cite[Lemma 25]{Sal2020}.
	\end{proof}
	
	\subsection{Restriction II}
	
	Define the auxiliary weight function $\nu_D: \bZ \to [0,\infty)$ by
	\[
	\nu_D(n) := 
	\frac{\varphi(D)}{\lam(D)}
	\sum_{\substack{p \le X \\
			p \equiv r_D \mmod D
			\\ h(p) = n \lam(D)}} 
	h'(p) \log p = 
	\frac{\varphi(D)}{\lam(D)}
	\sum_{\substack{z \le Z \\
			(Dz + r_D)\in\cP_X
			\\ h_D(z)=n}}
	h'(Dz + r_D) \log(Dz + r_D) .
	\]
	Observe that $\nu_D$ is supported on $h_D([Z])\subseteq [N]$. By the Siegel--Walfisz theorem, we have
	\begin{align*}
		\| \nu_D \|_1 
		&= \frac{\varphi(D)}{\lam(D)}
		\sum_{\substack{p \le X \\
				p \equiv r_D \mmod D}} 
		h'(p) \log p \\
		&\le \frac{\varphi(D)}{\lam(D)}
		\sum_{\substack{p \le X \\
				p \equiv r_D \mmod D}} h'(X) \log X
		\ll N.
	\end{align*}
	One can be more precise using partial summation, but we do not need to.
	
	In this subsection, we establish the following restriction estimate for $\nu_D$.
	
	\begin{prop}
		\label{prop6.7}
		Let $E > 2T$ be real, and let $\phi: \bZ \to \bC$ with $|\phi| \le \nu_D$. Then
		\[
		\int_\bT |\hat \phi(\alp)|^E \d \alp \ll_{h,E} N^{E-1}.
		\]
	\end{prop}
	
	Our approach is similar to that of Proposition \ref{prop6.1}, so we will not repeat all of the details. We introduce the auxiliary function
	\[
	\mu_D: \bZ \to [0,\infty),
	\qquad
	\mu_D(n) =
	\frac{ND}{X}
	\sum_{\substack{z \le Z
			\\ h_D(z) = n}} 1,
	\]
	noting that $\nu_D \ll (\log X) \mu_D$ pointwise. As a special case of \cite[Lemma 6.3]{CC}, we have the following sharp restriction estimate for $\mu_D$.
	
	\begin{lemma} 
		Let $E > 2T$, and let $\psi: \bZ \to \bC$ with $|\psi| \le \mu_D$. Then
		\[
		\int_\bT |\hat \psi(\alp)|^E 
		\d \alp \ll_{h,E} N^{E-1}.
		\]
	\end{lemma}
	
	The upshot is that if $v > 2T$ and $|\phi| \le \nu_D$ then
	\begin{equation} \label{eqn6.1}
		\int_\bT |\hat \phi(\alp)|^v 
		\d \alp
		\ll_v N^{v-1} (\log X)^v.
	\end{equation}
	To apply the general epsilon-removal lemma, we require major and minor arc bounds for
	\[
	\hat \nu_D(\alp)
	= \frac{\varphi(D)}{\lam(D)}
	\sum_{\substack{p \le X \\
			p \equiv r_D \mmod D}}
	h'(p) \log p \cdot
	e(\alp h(p) / \lam(D)).
	\]
	On the minor arcs, we infer the following analogue of Corollary \ref{cor5.3}, by essentially the same argument:
	\begin{equation}
		\label{eqn6.2}
		\hat \nu_D(\alp) \ll 
		X^d (\log X)^{-\sig_d} \qquad (\alp \in \fm).
	\end{equation}
	
	On the major arcs, we have \eqref{eqn5.2}, for some $q,a \in \bZ$. Define
	\[
	S_D(q,a) = \sum_{\substack{t \mmod{Dq} \\ (t,q) = 1 \\
			t \equiv r_D \mmod D}}
	e_q(ah(t)/\lam(D)).
	\]
	We infer the following analogue of Corollary \ref{cor5.4}, by essentially the same proof.
	
	\begin{lemma} Let $c$ be a small, positive constant, small in terms of $C_d$. Suppose 
		$(\alp, q, a) \in 
		\bR \times \bN \times \bZ$ 
		with \eqref{eqn5.2}, and put $\bet = \alp - a/q$. Then
		\[
		\hat \nu_D(\alp) =
		\frac{\varphi(D)}
		{\varphi(Dq)} 
		S_D(q,a) I(\bet)
		+ O(Ne^{-c \sqrt{\log X}}).
		\]
	\end{lemma}
	
	It follows from 
	Lemma \ref{lem2.2} that
	\[
	S_D(q,a) \ll_{h,\eps} q^{\eps-1/d}
	\varphi(q).
	\]
	Thus, by incorporating \eqref{eqn5.6}, we arrive at the following variant of Corollary \ref{cor5.6}:
	\begin{equation}
		\label{eqn6.3}
		\hat \nu_D(\alp)
		\ll_{h,\eps} q^{\eps - 1/d}
		\min \{ N, \| \alp - a/q \|^{-1} \} 
		+ O(Ne^{-c \sqrt{\log X}}).
	\end{equation}
	Equipped with the bounds \eqref{eqn6.1}, \eqref{eqn6.2}, and \eqref{eqn6.3}, Proposition \ref{prop6.7} now follows from the general epsilon-removal lemma \cite[Lemma 25]{Sal2020}.
	
	\section{The transference principle}\label{sec7}
	
	In this section, we are finally ready to use transference to deduce Theorem \ref{thm2.5} from Theorem \ref{thm2.7}.
	We start with some notation and a preparatory lemma.
	
	For finitely-supported $f_1, \ldots, f_s, g_1, \ldots, g_t: \bZ \to \bR$, define
	\[
	\Phi(f_1,\ldots,f_s;
	g_1, \ldots, g_t) = 
	\sum_{L_1(\bn) = L_2(\mathbf{m})}
	f_1(n_1) \cdots f_s(n_s)
	g_1(m_1) \cdots g_t(m_t).
	\]
	We frequently make use of the abbreviations
	\[
	\Phi(f_1,\ldots,f_s; g)
	= \Phi(f_1,\ldots,f_s; g, \ldots, g),
	\qquad
	\Phi(f; g) =
	\Phi(f,\ldots,f; g, \ldots, g).
	\]
	Given finite sets of integers $A$ and $B$, we also write
	$\Phi(A;g) = \Phi(1_A; g)$, and similarly for the expressions $\Phi(f;B)$ and $\Phi(A;B)$.
	
	\begin{lemma}[Fourier control]
		\label{lem7.1}
		Let $f_1, \ldots, f_s, g: \bZ \to \bR$ be finitely supported. If
		\[
		|f_j| \le \nu + 1_{[N]}
		\quad (1 \le j \le s),
		\qquad
		|g| \le \nu_D,
		\]
		then
		\[
		\Phi(f_1,\ldots,f_s; g)
		\ll N^{s+t-1} \prod_{j \le s}
		(\| \hat f_j \|_\infty 
		/ N)^{1/(2s+2t)}.
		\]
	\end{lemma}
	
	\begin{proof} Following the proof of \cite[Lemma 7.1]{CC} yields
		\begin{align*}
			&|\Phi(f_1,\ldots,f_s;g)|
			\le
			\left(
			\int_\bT |\hat g(\alp)|^{s+t} \d \alp
			\right)^{t/(s+t)} \cdot 
			\prod_{j \le s}
			\left( \| \hat f_j \|_\infty^{1/2}
			\int_\bT |\hat f_j(\alp)|^{s+t-1/2}
			\d \alp
			\right)^{1/(s+t)}.
		\end{align*}
		Propositions \ref{prop6.2} and \ref{prop6.7} now give
		\begin{align*}
			\Phi(f_1,\ldots,f_s;g) &\ll
			N^{(s+t-1)t/(s+t)}
			\prod_{j \le s}
			\left(
			\| \hat f_j \|_\infty^{1/(2s+2t)}
			N^{(s+t-3/2)/(s+t)} \right) \\
			&= N^{s+t-1}
			\prod_{j \le s} (\| \hat f_j \|_\infty / N)^{1/(2s+2t)}.
		\end{align*}
	\end{proof}
	
	\begin{proof}[Proof of Theorem \ref{thm2.5} given Theorem \ref{thm2.7}]
		
		Fix $\del, h, r, L_1,$ and $L_2$. The implied constants are henceforth allowed to depend on all of these parameters. 
		Let $\tilde \del$ be small in terms of the fixed parameters, and let $w \in \bN$ be large in terms of them. We insist that $\tilde \del$ is an integer power of 2, and that $\tilde \del \gg 1$, so that dependence on $\tilde \del$ is subsumed by dependence on the fixed parameters.
		
		Given sufficiently large $X\in\N$, we define $D,N,W,Z$ as in \S\ref{sec3}. For the purposes of proving Theorem~\ref{thm2.5}, we may assume that $Z\in\N$. Indeed, assuming $X$ is sufficiently large relative to $D$, any $A\subseteq\cP_X$ with $|A|\geqslant\delta |\cP_X|$ must satisfy $|A\cap[X-D]|>(\delta/2)|\cP_X|$. Thus, by replacing $(\delta,A,X)$ with $(\delta/2,A\cap[X-a],X-a)$ for some $a\in[D]$ such that $D\mid(X-a-r_D)$, we can assume that $D\mid(X-r_D)$, whence $Z\in\N$.
		
		Set
		\[
		\tilde \cC_i = 
		\{ z \in [Z]: r_D + Dz \in \cC_i \}
		\qquad (1 \le j \le r).
		\]
		By Theorem \ref{thm2.7}, there exists $k \in [r]$ such that every $\tilde \cA \subseteq [N]$ with $|\tilde \cA| \ge \tilde \del N$ satisfies
		\[
		\# \{ (\bn, \bz) \in \tilde \cA^s \times \tilde \cC_k^t:
		L_1(\bn) = L_2(h_D(\bz)) \}
		\gg N^{s-1} \left(
		\frac{DZ}{\varphi(D)\log Z} 
		\right)^t.
		\]
		By a simple counting argument, the number of solutions counted here for any given value of $z_t$ is 
		$O(|\tilde C_k| N^{s-1} 
		(Z/\log Z)^{t-1})$, whence $|\tilde \cC_k| \gg Z/\log Z$. Thus,
		\[
		|\cC_k| \ge |\tilde \cC_k| 
		\gg \frac{Z}{\log Z} \gg_w 
		\frac{X}{\log X}.
		\]
		
		Define $\cA$ by \eqref{eqn3.2},
		and define
		\[
		f = \nu 1_\cA,
		\qquad
		g_i(n) = 
		\frac{N \varphi(D) \log Z}{DZ}
		\sum_{\substack{
				z \in \tilde C_i \\ h_D(z) = n}}1
		\quad (1 \le i \le r).
		\]
		By Lemma \ref{lem5.1} and the dense model lemma \cite[Theorem 5.1]{Pre2017}, there exists a function $f_0$ such that
		\[
		0 \le f_0 \le 1_{[N]},
		\qquad
		\| \hat f - \hat f_0 \|_\infty
		\ll (\log w)^{-3/2} N.
		\]
		For $\ell \in [s]$, we write $\bu^{(\ell)} = (u_1^{(\ell)},
		\ldots, u_s^{(\ell)})$, where
		\[
		u_j^{(\ell)} = 
		\begin{cases}
			f_0, &\text{if } j < \ell \\
			f-f_0, &\text{if } j = \ell \\
			f, &\text{if } j > \ell.
		\end{cases}
		\]
		The telescoping identity and Lemma \ref{lem7.1} now give
		\begin{equation}\label{eqn7.1}
			\Phi(f; g_i) - \Phi(f_0; g_i)
			= \sum_{\ell \le s}
			\Phi(\bu^{(\ell)}; g_i)
			\ll (\log w)^{-3/(4s + 4t)}
			N^{s+t-1} \quad (1 \le i \le r).
		\end{equation}
		By Lemma \ref{lem3.1}, we have
		\[
		\sum_{n \in \bZ} f(n) \gg N.
		\]
		Since $\hat f(0) - \hat f_0(0) \ll (\log w)^{-3/2}$, and $w$ is large, we also have
		\[
		\sum_{n \in \bZ} f_0(n) \gg N.
		\]
		
		Let $c$ be a small, positive constant which depends only on the fixed parameters. Setting
		\[
		\tilde \cA = \{ n \in \bZ:
		f_0(n) \ge c \},
		\]
		the popularity principle \cite[Exercise 1.1.4]{TV2006} allows us to obtain the lower bound $|\tilde \cA| \ge \tilde \del N$. Now Theorem \ref{thm2.7} gives
		$
		\Phi(\tilde \cA; g_k) \gg
		N^{s+t-1}.
		$
		Since $0\leqslant c1_{\tilde{A}}\leqslant f_0$, it follows that $\Phi(f_0; g_k) \gg N^{s+t-1}$. Taking $w$ sufficiently large, we infer from (\ref{eqn7.1}) that 
		\[
		\Phi(f; g_k) \gg N^{s+t-1}.
		\]
		Finally, since $N\leqslant h(X)\asymp X^d$, we conclude that
		\begin{align*}
			&\# \{(\bx, \by) \in A^s \times \cC_k^t: L_1(h(\bx)) = L_2(h(\by)) \}
			\\
			&\ge \| f \|_\infty^{-s}
			\| g_k \|_\infty^{-t} \Phi(f; g_k)
			\gg_w 
			\left(
			\frac{N \log X}{X}
			\right)^{-s}
			\left( \frac{N \log X}{X}
			\right)^{-t}
			N^{s+t-1} \\
			&\gg_w
			\frac{X^{s+t-d}}{(\log X)^{s+t}}.
		\end{align*}
		By specifying that $w=O_{\delta,h,r,L_1,L_2}(1)$, this completes the proof.
	\end{proof}
	
	\section{Arithmetic regularity}\label{sec8}
	
	Our only remaining task is to prove Theorem \ref{thm2.7}. We are therefore interested in counting solutions to the linearised equation (\ref{eqn2.7}). We seek a colour class $\cC_k$ such that there are many solutions $(\bn,\bz)$ to (\ref{eqn2.7}) with $\bn\in\cC_k^t$ and $\bz\in\cA^s$ for some arbitrary dense set $\cA\subseteq [N]$. 
	
	Following \cite{Pre2021} and \cite{CC}, we begin by simplifying the statement of Theorem \ref{thm2.7}. Rather than assert the existence of a colour class $\cC_k$ which gives many solutions with respect to all dense sets $\cA$, we instead consider a finite collection of dense sets $\cA_1,\ldots,\cA_r\subseteq[N]$ and seek a colour class $\cC_k$ such that, for all $i\in[r]$, there are many solutions to (\ref{eqn2.7}) with $(\bn,\bz)\in\cA_i^s\times\cC_k^t$. This leads to the following version of Theorem \ref{thm2.7}.
	
	\begin{thm} \label{thm8.1}
		Let $r$ and $d\geqslant 2$ be positive integers, and let $0<\delta<1$. 
		Let $h$ be an integer polynomial of degree $d$ which is intersective of the second kind.
		Let $s \ge 1$ and $t \ge 0$ be integers such that $s + t \ge s_0(d)$. Let
		\[
		L_1(\bx) \in \bZ[x_1,\ldots,x_s], \qquad L_2(\by) \in \bZ[y_1,\ldots,y_t] 
		\]
		be non-degenerate linear forms such that $L_1(1,\ldots,1) = 0$ and $\gcd(L_1)=1$.
		Then there exists $\eta=\eta(d,\delta,L_1,L_2)\in(0,1)$ such that the following is true.
		Let $D, Z \in \bN$ satisfy $Z \ge Z_0(D, h, r, \del, L_1, L_2)$, and set $N:=h_D(Z)$. 
		Suppose $\cA_1,\ldots,\cA_r\subseteq[N]$ satisfy $|\cA_i|\geqslant\delta N$ for all $i\in[r]$. If
		\[
		[\eta Z,Z]\cap\{ z \in [Z]: r_D + Dz \in \cP \}
		= \cC_1 \cup \cdots \cup \cC_r,
		\]
		then there exists $k \in [r]$ such that
		\[
		\# \{ (\bn,\bz) \in \cA_i^s \times \cC_k^t: L_1(\bn) = L_2(h_D(\bz)) \} \gg 
		N^{s-1} 
		\left(\frac{DZ}{\varphi(D)\log Z} \right)^t \qquad(1\leqslant i\leqslant r).
		\]
		The implied constant may depend on $h, L_1, L_2, r, \delta$.
	\end{thm}
	
	\begin{proof}[Proof of Theorem \ref{thm2.7} given Theorem \ref{thm8.1}]
		In view of Proposition \ref{prop2.9}, it is enough to prove Theorem \ref{thm2.7} under the assumption that $\gcd(L_1)=1$. We claim that Theorem \ref{thm2.7} holds with the same quantities $Z_0(D,h,r,\delta,L_1,L_2)$, $\eta(d,\delta,L_1,L_2)$, and the same implicit constant $C=C(h,L_1,L_2,r,\delta)$ appearing in the final bound. Suppose for a contradiction that this is false. For each $k\in[r]$, we can then find $\cA_k\subseteq[N]$ with $|\cA_k|\geqslant\delta N$ such that
		\[
		\# \{ (\bn,\bz) \in \cA_k^s \times \cC_k^t: L_1(\bn) = L_2(h_D(\bz)) \} <
		CN^{s-1} 
		\left(\frac{DZ}{\varphi(D)\log Z} \right)^t.
		\]
		Applying Theorem \ref{thm8.1} to the collection of dense sets $\cA_1,\ldots,\cA_r$ delivers a contradiction.
	\end{proof}
	
	By taking $Z$ sufficiently large relative to $\eta$ in Theorem \ref{thm8.1}, we may assume that $h_D$ is positive and strictly increasing on the real interval $[\eta Z, Z]$. We can then define a function $\cQ_{Z}=\cQ_{Z;\eta,h_D}:\Z\to\R$ by
	\begin{equation*}
		\cQ_Z(t) :=
		\begin{cases}
			z, &\text{if }z\in[\eta Z, Z] \text{ satisfies } t=h_D(z) \\
			0, &\text{otherwise}.
		\end{cases}
	\end{equation*}
	Notice that $\cQ_Z$ is supported on $h_D([\eta Z, Z])\subseteq [N]$, and the restriction of $\cQ_Z$ to $h_D([\eta Z, Z])$ defines a bijection from $h_D([\eta Z, Z])$ to $[\eta Z, Z]$. Hence, given functions $f_1,\ldots,f_s:\Z\to\R$ supported on $[N]$, and $g:\Z\to\R$ supported on $[\eta Z, Z]$, we have
	\begin{equation*}
		\label{eqnCountingH}
		\Phi(f_1,\ldots,f_s;
		g\circ\cQ_Z) 
		= \sum_{L_1(\bn) = L_2(h_D(\bz))}
		f_1(n_1) \cdots f_s(n_s)
		g(z_1) \cdots g(z_t).
	\end{equation*}
	
	\begin{lemma}
		[Arithmetic regularity lemma]
		\label{lem8.2}
		Let $r\in\N$, $\sig>0$, and let $\cF:\R_{\geqslant 0}\to\R_{\geqslant 0}$ be a monotone increasing function. Then there exists a positive integer $K_{0}(r;\sig,\cF)\in\N$ such that the following is true. Let $N\in\N$ and $f_{1},\ldots,f_r:[N]\to[0,1]$. Then there is a positive integer $K\leqslant K_{0}(r;\sig,\cF)$ and a phase $\btheta\in\T^{K}$ such that, for every $i\in[r]$, there is a decomposition
		\begin{equation*}
			f_{i}=f_{\str}^{(i)}+f_{\sml}^{(i)}+f_{\unf}^{(i)}
		\end{equation*}
		of $f_{i}$ into functions $f_{\str}^{(i)},f_{\sml}^{(i)},f_{\unf}^{(i)}:[N]\to[-1,1]$ with the following stipulations.
		\begin{enumerate}[\upshape(I)]
			\item\label{itemNon} The functions $f_{\str}^{(i)}$ and $f_{\str}^{(i)}+f_{\sml}^{(i)}$ take values in $[0,1]$.
			\item\label{itemL2} The function $f_{\sml}^{(i)}$ obeys the bound $\lVert f_{\sml}^{(i)}\rVert_{L^{2}(\Z)}\leqslant\sig\lVert 1_{[N]}\rVert_{L^{2}(\Z)}$.
			\item\label{itemUnf} The function $f_{\unf}^{(i)}$ obeys the bound $\lVert \hat{f}_{\unf}^{(i)}\rVert_{\infty}\leqslant\lVert \hat{1}_{[N]}\rVert_{\infty}/\cF(K)$.
			\item\label{itemSum} The function $f_{\str}^{(i)}$ satisfies $\sum_{m=1}^{N}(f_{i}-f_{\str}^{(i)})(m)=0$.
			\item\label{itemStr} There exists a $K$-Lipschitz function $F_{i}:\T^{K}\to[0,1]$ such that $F_{i}(x\btheta )=f_{\str}^{(i)}(x)$ for all $x\in[N]$. 
		\end{enumerate}
	\end{lemma}
	\begin{proof}
		This is \cite[Lemma 8.3]{CC}.
	\end{proof}
	
	Applying this to a given function $f$ allows us to write $\Phi(f; g)$ as the sum of $\Phi(f_\str + f_\sml; g)$ and $2^s -1$ terms $\Phi(f_1,\ldots,f_s;g)$, where at least one of the $f_i$ equals $f_{\unf}$ and the rest are equal to $f_{\str}+f_{\sml}$. As is typical in applications of the arithmetic regularity lemma, we expect the term $\Phi(f_{\str}+f_{\sml};g)$ to provide the main contribution, whilst the remaining terms should be asymptotically negligible. This prediction is verified by combining Property~(\ref{itemUnf}) of Lemma~\ref{lem8.2} with our Fourier control result (Lemma \ref{lem7.1}).
	
	To carry out this strategy of removing the contribution of $f_\unf$, we need to perform a minor technical manoeuvre. Applying Lemma \ref{lem7.1} requires us to bound the function $g$ appearing in $\Phi(f; g)$ in terms of $\nu_D$. To achieve a sharp asymptotic lower bound for the number of solutions, we desire a bound of the form 
	$g\lVert \nu_D\rVert_{\infty}
	\ll \nu_D$. This is the method we used in \cite[Lemma 8.4]{CC}, only with $\mu_D$ in place of $\nu_D$. However, this relied on the fact that $\mu_D$ is constant on its support, whilst $\nu_D$ is not. In particular, if $g$ is the indicator function of a colour class, then $g(z)\lVert \nu_D\rVert_{\infty}$ could be asymptotically larger than $\nu_D(z)$ for small $z$.
	
	To overcome this issue, we restrict attention from $[Z]$ to $[\eta Z, Z]$, for some sufficiently small $\eta>0$. On this latter interval, the function $\nu_D$ does not vary too much. This is made precise by the following lemma, which is a variation of \cite[Lemma 8.12]{CC}.
	
	\begin{lemma}
		\label{lem8.3}
		Let $P$ be a real polynomial of degree $d\in\N$ with positive leading coefficient. Then there exists a positive integer $M_0(P)$ such that the following is true. For all $\eta\in(0,1)$, if $x\in\R$ satisfies $x\geqslant \eta^{-1}M_{0}(P)$, then     
		\begin{equation*}
			\eta^d P(x) \leqslant 3 P(\eta x) \leqslant 9\eta^d P(x).
		\end{equation*}
	\end{lemma}
	\begin{proof}
		Let $\ell_P>0$ be the leading coefficient of $P$. We can then find $M_0(P)\in\N$ such that 
		\begin{equation*}
			\ell_P x^d \leqslant 2P(x) \leqslant 3\ell_P x^d
		\end{equation*}
		holds for all real $x\geqslant M_0(P)$. In particular, if $x\geqslant \eta^{-1}M_{0}(P)$, then  
		\begin{equation*}
			\frac{\eta^d}{3} \leqslant \frac{P(\eta x)}{P(x)} \leqslant 3\eta^d.
		\end{equation*}
	\end{proof}
	
	\begin{lemma}
		[Removing $f_{\unf}$]
		\label{lem8.4}
		Let $f:\Z\to[0,1]$ be supported on $[N]$. Let $\eta,\sig>0$, and let $\cF:\R_{\geqslant 0}\to\R_{\geqslant 0}$ be a monotone increasing function. Let $f_{\str},f_{\sml},f_{\unf}$ be the functions obtained upon applying Lemma \ref{lem8.2} to $f$. Then for any $g:\Z\to[0,1]$ supported on the set
		\begin{equation*}
			\{ z\in [Z]: r_D + Dz\in\cP\}\cap[\eta Z,Z],
		\end{equation*}
		we have
		\begin{equation*}
			\lvert\Phi(f;g\circ\cQ_Z)-\Phi(f_{\str}+f_{\sml};g\circ\cQ_Z)\rvert\ll_{h,\eta,D} N^{s-1}\left(\frac{DZ}{\varphi(D)\log Z} \right)^t \cF(K)^{-1/(2s+2t)}.
		\end{equation*}
	\end{lemma}
	
	\begin{proof}
		Note that $|f|\leqslant 1_{[N]}$ and $\lVert \hat{f}\rVert_{\infty} \leqslant N$.
		Thus, by using a telescoping identity, as in the derivation of \eqref{eqn7.1}, Lemma \ref{lem7.1} informs us that
		\begin{equation*}
			\lvert\Phi(f;G)-\Phi(f_{\str}+f_{\sml};G)
			\rvert \ll N^{s+t-1}
			\cF(K)^{-1/(2s+2t)}
		\end{equation*}
		holds for any $G:\Z\to\R$ such that $|G|\leqslant \nu_D$. Taking $G=\xi(g\circ\cQ_Z)$ for some $\xi>0$, we deduce that
		\begin{equation*}
			\lvert
			\Phi(f;g\circ\cQ_Z) -\Phi(f_{\str}+f_{\sml};
			g\circ\cQ_Z) \rvert
			\ll N^{s-1}(N/\xi)^t
			\cF(K)^{-1/(2s+2t)}.
		\end{equation*}
		
		To complete the proof, it remains to find $\xi>0$ with $\xi|g\circ\cQ_Z|\leqslant \nu_D$ such that
		\begin{equation*}
			N/\xi \ll \frac{DZ}{\varphi(D)\log Z}.
		\end{equation*}
		Set
		\begin{equation*}
			B = \{ n \in 
			h_D([\eta Z,Z]\cap\N): n=h_D(z), 
			\quad r_D + Dz \in\cP\}.
		\end{equation*}
		Observe that, for all $n=h_D(z)\in B$, we have
		\begin{equation*}
			\lambda(D)\nu_D(n) \geqslant \varphi(D)
			h'(D\eta Z + r_D)
			\log(D\eta Z + r_D).
		\end{equation*}
		By Lemma \ref{lem8.3}, if $Z$ is sufficiently large relative to $h$, $\eta$, and $D$, then
		\begin{equation*}
			DZh'(D\eta Z +r_D) \log(D\eta Z +r_D) \gg h(X)\log Z = \lambda(D) N\log Z. 
		\end{equation*}
		Since $g\circ\cQ_Z$ is supported on $B$ and takes values in $[0,1]$, we conclude that there exists $c \gg 1$ such that $\xi:=c(DZ)^{-1}\varphi(D)N
		\log Z$ has all the required properties.
	\end{proof}
	
	\section{Prime polynomial Bohr sets}\label{sec9}
	
	The final step of the proof of Theorem \ref{thm8.1} is to obtain a lower bound for the main term $\Phi(f_{\str}+f_{\sml};g\circ\cQ_Z)$. By our assumption that the coefficients of $L_1$ are coprime, we can find $\bv\in\Z^s$, which depends only on $L_1$, such that $L_1(\bv) =1$. We can therefore write
	\begin{equation*}
		\Phi(f_1,\ldots,f_s;g\circ\cQ_Z) = \sum_{\bz\in\Z^t}g(z_1) \cdots g(z_t)\Psi_{\bz}(f_1,\ldots,f_s),
	\end{equation*}
	where we have introduced the auxiliary counting operator
	\begin{equation*}
		\Psi_{\bz}(f_1,\ldots,f_s) := \sum_{L_1(\bn)=0}\prod_{i=1}^{s}f_i(n_i + v_iL_2(h_D(\bz))).
	\end{equation*}
	
	Our goal is to show that there is a large supply of $\bz\in\Z^t$ for which $\Psi_\bz(f_\str + f_\sml)$ is asymptotically as large as possible. This is accomplished by choosing the $z_i$ to lie in a set of `almost-periods' for $f_\str$. These are known as \emph{polynomial Bohr sets} and take the form
	\begin{equation*}
		\{n\in\N:\lVert Q(n)\balp\rVert<\rho\}
		=\bigcap_{i=1}^{K}
		\{n\in\N: \lVert Q(n)\alpha_{i}\rVert
		<\rho\},
	\end{equation*}
	for some $\rho>0$, $K\in\N$, $\balp\in\T^K$, and $Q(x) \in \Z[x]$.
	
	\begin{lemma}
		[Lower bound for $\Psi_{\bz}(f_{\str}+f_{\sml})$]
		\label{lem9.1}
		For all $\delta>0$, there exist positive constants $c_{1}(\delta)=c_{1}(L_1,L_2;\delta)>0$ and $\eta_0 = \eta_0(d,L_1,L_2,\delta)>0$ such that the following is true. Suppose $f:\Z\to[0,1]$ is supported on $[N]$ and satisfies $\lVert f\rVert_1\geqslant \delta N$. Given $\sig \in (0,1]$ and a monotone increasing function $\cF:\R_{\geqslant 0}\to\R_{\geqslant 0}$, let $f_{\str}$, $f_{\sml}$, $K$ and $\btheta$ be as given by applying Lemma~\ref{lem8.2} to $f$. If $\bz\in[\eta Z]^{t}$ satisfies
		\begin{equation*}
			\lVert h_D(z_i)\theta_j\rVert < \sigma/K \qquad (1\leqslant i\leqslant t, \quad 
			1\leqslant j\leqslant K),
		\end{equation*}
		then
		\begin{equation*}
			\Psi_{\bz}(f_{\str}+f_{\sml}) \geqslant \left(c_{1}(\delta)- O_{L_1,L_2}(\sig)\right)N^{s-1}. 
		\end{equation*}
		In particular, if $\sigma$ is sufficiently small relative to $(d,L_1,L_2,\delta)$, then
		\begin{equation*}
			\Psi_{\bz}(f_{\str}+f_{\sml}) \gg_{L_1,L_2,\delta}N^{s-1}. 
		\end{equation*}
	\end{lemma}
	\begin{proof}
		This is \cite[Lemma 8.13]{CC} with $\rho=\sigma/K$.
	\end{proof}
	
	Recall that we seek solutions to the linearised equation (\ref{eqn2.7}) with $r_D + Dz_j$ prime for all $j$. In view of Lemma \ref{lem9.1}, we are therefore interested in sets of the form
	\begin{equation}\label{eqn9.1}
		\cB_h(\balp,\rho):=\{p\in\cP: p\equiv r_D\mmod{D}, \;\lVert \balp h(p)/\lam(D) \rVert < \rho \}.
	\end{equation}
	The main purpose of this section is to establish the following density bounds for these \emph{prime polynomial Bohr sets}.
	
	\begin{thm}
		\label{thm9.2}
		Let $K,D,d\in\N$, $\rho>0$, and let $h$ be an integer polynomial of degree $d$ which is intersective of the second kind. Then there exists a positive real number $\Delta(\rho)=\Delta(h,K;\rho)$ and a positive integer $P_1=P_1(D,h,K,\rho)$ such that the following is true for all $P\geqslant P_1$. If $\balp\in\T^K$, then
		\begin{equation*}
			\sum_{p\in\cB}\log p \geqslant\frac{ \Delta(\rho)P}{\varphi(D)},
		\end{equation*}
		where $\cB=[P]\cap\cB_h(\balpha,\rho)$. Moreover, we may take
		\begin{equation*}
			\Delta(h,1;\rho) \gg_{h,\eps} \rho^{d+3+\eps}, \qquad \Delta(h,K;\rho) 
			\gg_{h,K, \eps} \:
			\rho^{3 + K(d+\eps)} \Delta\left(h,K-1;\frac{\rho^{2}}{2K^2}\right) \quad (K>1).
		\end{equation*}
	\end{thm}
	
	We demonstrate the utility of this result by using it to complete the proof of Theorem \ref{thm8.1}.
	
	\begin{proof}
		[Proof of Theorem \ref{thm8.1} given Theorem \ref{thm9.2}]
		As usual, we fix the parameters
		\begin{equation*}
			\delta,h,r,L_1,L_2
		\end{equation*}
		appearing in the statement of Theorem \ref{thm8.1}. Unless specified otherwise, we allow all forthcoming implicit constants to depend implicitly on these parameters. 
		Let $\sig,\eta\in(0,1)$ and let $\cF:\R_{\geqslant 0}\to\R_{\geqslant 0}$ be a monotone increasing function, all three of which depend only on the fixed parameters. Let $D\in\N$, and assume throughout this proof that $Z\in\N$ and $N:=h_D(Z)$ are sufficiently large with respect to all of these quantities.
		
		For each $i\in[r]$, let $\cA_i\subseteq [N]$ with $|\cA_i|\geqslant\delta N$. Suppose we have an $r$-colouring 
		\[
		\{ z \in [\eta Z,Z]: 
		r_D + Dz \in \cP \}
		= \cC_1 \cup \cdots \cup \cC_r.
		\]
		In the notation of the previous section, our goal is to find $k\in\N$ such that
		\begin{equation*}
			\Phi(1_{\cA_i};1_{\cC_k}\circ\cQ_Z) \gg 
			N^{s-1} 
			\left(\frac{DZ}{\varphi(D)\log Z} \right)^t \qquad(1\leqslant i\leqslant r).
		\end{equation*}
		
		Lemma \ref{lem8.2} provides us with decompositions
		\begin{equation*}
			1_{\cA_i} = f_\str^{(i)} + f_\sml^{(i)} + f_\unf^{(i)} \qquad (1\leqslant i\leqslant r),
		\end{equation*}
		along with a positive integer $K\ll_{\sigma,\cF} 1$ and a phase $\btheta\in\T^K$ with the properties described therein. Let $\eta_0=\eta_0(d,L_1,L_2,\delta)$ be as defined in Lemma \ref{lem9.1}, and let
		\begin{equation*}
			\Omega :=\{ z\in[\eta Z, \eta_0 Z]: r_D + Dz \in \cP, \: \lVert h_D(z)\btheta\rVert < \sigma/K\}\subseteq \cC_1\cup\cdots\cup\cC_r.
		\end{equation*}
		By choosing $\sigma$ sufficiently small, Lemma \ref{lem9.1} informs us that
		\begin{equation*}
			\Phi(f_\str^{(i)}+f_\sml^{(i)}; 1_{\cC_j}\circ\cQ_Z) \gg N^{s-1}|\Omega\cap\cC_j|^t \qquad (1\leqslant i,j\leqslant r).
		\end{equation*}
		We now claim that, for an appropriate choice of $\eta<\eta_0$, the set $\Omega$ satisfies
		\begin{equation}
			\label{eqn9.2}
			|\Omega| \gg_K \frac{DZ}{\varphi(D)\log Z}.
		\end{equation}
		Assume for the moment that this is true. By the pigeonhole principle, we can choose $k\in[r]$ such that $r|\Omega\cap\cC_k|\geqslant|\Omega|$, whence
		\begin{equation*}
			\Phi(f_\str^{(i)}+f_\sml^{(i)}; 1_{\cC_k}\circ\cQ_Z) \gg_K
			N^{s-1} 
			\left(\frac{DZ}{\varphi(D)\log Z} \right)^t \qquad (1\leqslant i\leqslant r).
		\end{equation*}
		Incorporating Lemma \ref{lem8.4} furnishes the bound
		\begin{equation*}
			\Phi(1_{\cA_i};1_{\cC_k}\circ\cQ_Z) \gg N^{s-1}\left(\frac{DZ}{\varphi(D)\log Z} \right)^t \left(c(K) - \cF(K)^{-1/(2s+2t)}\right),
		\end{equation*}
		for some positively-valued function $c(K)>0$ whose value depends only on the fixed parameters and $K$. Specifying $\cF:\R_{\geqslant 0}\to\R_{\geqslant 0}$ to be a monotone increasing function which obeys
		\begin{equation*}
			2\cF(y)^{-1/(2s+2t)} \leqslant c(y) \qquad (y\in\N)
		\end{equation*}
		then finishes the proof of Theorem \ref{thm8.1}, subject to our claim.
		
		It remains to establish \eqref{eqn9.2}. Let
		\begin{equation*}
			P = \eta_0 DZ + r_D.
		\end{equation*}
		Let $\rho=\sigma/K$,
		and for each $\xi\in(0,1)$ put
		\begin{equation*}
			\cD_\xi := \cP\cap\{p\in[\xi P,P]:p\equiv r_D\mmod{D}, \;\lVert \btheta h(p)/\lam(D) \rVert < \rho \} = [\xi P,P]\cap\cB_h(\btheta,\rho).
		\end{equation*}
		By Theorem \ref{thm9.2} and the Siegel-Walfisz theorem, there exists $\xi\in(0,1)$ with $\xi\gg\Delta(\rho)$ such that
		\begin{equation*}
			\sum_{p\in\cD_\xi}\log p \geqslant \frac{\Delta(\rho)P}{2\varphi(D)} \gg_\rho \frac{P}{\varphi(D)}.
		\end{equation*}
		Choosing $\eta = \eta_0\xi/2$, we ensure that the injective function $y\mapsto (y-r_D)/D$ maps $[\xi P,P]$ into $[\eta Z, \eta_0 Z]\subseteq [\eta Z, Z]$ and maps $\cD_\xi$ into $\Omega$.
		Since $\rho=O_K(1)$, we therefore conclude that
		\begin{equation*}
			|\Omega|\geqslant \sum_{p\in\cD_\xi}
			\frac{\log p}{\log P} \gg_K \frac{DZ}{\varphi(D)\log Z},
		\end{equation*}
		as claimed.
	\end{proof}
	
	\subsection{Exponential sums}
	
	To study prime polynomial Bohr sets, we are interested in exponential sums over primes of the form
	\begin{equation*}
		\sum_{\substack{p \le P \\
				p \equiv r_D \mmod D}}
		e \left(
		\frac{h(p) \theta}{\lam(D)} \right) \log p,
	\end{equation*}
	where $\theta\in\T$ and $h$ is intersective of the second kind. L\^{e} and Spencer \cite{LS2014} analysed properties of sums of this form to obtain estimates for the smallest element of a prime polynomial Bohr set (\ref{eqn9.1}), showing in particular that these sets are always non-empty. Our goal is to obtain a lower bound for the densities of these Bohr sets which does not depend on the choice of phase $\balpha$.
	
	Following \cite{LS2014}, our argument begins with the observation that if the prime polynomial Bohr set (\ref{eqn9.1}) has few elements, then we can construct a corresponding exponential sum which is particularly large. This is elucidated by the following lemma, which is a consequence of a much more general result of Harman \cite{Har1993}.
	
	\begin{lemma}\label{lem9.3}
		Let $D,K,P\in\N$. Let $h$ be an integer polynomial of degree $d\in\N$ which is intersective of the second kind. Define $r_D$ and $\lambda(D)$ as in \S\ref{subsec2.4}. Let $\rho\in (0,1)$, $\balpha\in\T^K$, and
		\begin{equation*}
			\cC = \cC_h(\balpha,\rho):= \cP\cap\left \{ p \le P: p \equiv r_D \mmod D, 
			\left \| \frac{h(p)}{\lam(D)} \balp \right \| \geqslant \rho \right \}.
		\end{equation*}
		Then there exists $\mathbf{m}\in\Z^K$ with $0<\lVert \mathbf{m}\rVert_{\infty} \leqslant K\rho^{-1}$ such that
		\begin{equation*}
			(2K + 1)^K\left\lvert
			\sum_{p\in\cC}
			e \left(\frac{h(p) \mathbf{m}\cdot\balpha}{\lam(D)} \right) 
			\log p\right\rvert \geqslant \frac{\rho^K}{4K^2 - 1}
			\sum_{p\in\cC}\log p.
		\end{equation*}
	\end{lemma}
	
	\begin{proof}
		If $\rho > 1/2$, then $\cC$ is empty and both sides of the desired inequality equal zero. If $\rho \le 1/2$, then the result follows from the contrapositive of \cite[Corollary to Lemma 5]{Har1993} and the pigeonhole principle.
	\end{proof}
	
	For the purpose of proving Theorem \ref{thm9.2}, we may assume that the Bohr set $\cB_h(\balpha,\rho)$ has too few elements, in a manner that will be clarified in due course. Then we can apply Lemma~\ref{lem9.3} to find some $L\ll_K \rho^{-K}$ and $\mathbf{m}\in\Z^K$ of bounded size such that $\theta = \mathbf{m}\cdot\balpha$ satisfies
	\begin{equation}
		\label{eqn9.3}
		\left|
		\sum_{\substack{p \le P \\
				p \equiv r_D \mmod D}}
		e \left(
		\frac{h(p) \theta}{\lam(D)} \right) \log p \right|
		\geqslant \frac{P}{L\varphi(D)}.
	\end{equation}
	
	Our next task is to investigate the consequences of (\ref{eqn9.3}). As discussed in \S\ref{sec4}, an exponential sum being large is indicative of the phase $\theta$ exhibiting `major arc' behaviour, meaning that $\theta$ is well-approximated by a rational number with small denominator. This is made precise in the following lemma.
	
	\begin{lemma} [Low major arc]
		\label{lem9.4}
		Suppose $L,P\in\N$ and $\theta\in\R$ satisfy (\ref{eqn9.3}). Assume that $P$ is sufficiently large relative to $D,h$ and $L$. Then there exists $q \in \bN$ such that
		\[
		q \ll_{h,\eps} L^{d+\eps},
		\qquad \| q\theta \| \ll_h qL^d \lam(D) / P^d.
		\]
	\end{lemma}
	
	\begin{proof}
		By \eqref{eqn9.3} and Lemma \ref{lem4.2}, we can find $r\in\N$ and $b\in\Z$ such that
		\[
		\max \left \{ r, P^d 
		\left| 
		r \frac{\theta}
		{\lam(D)} - b
		\right|
		\right \}
		\ll (\log P)^{2 C_d}.
		\]
		Let $q \in \bN$ and $a \in \bZ$ with
		\[
		(a,q) = 1, \qquad
		\frac{\lam(D) b}{r} = 
		\frac a q.
		\]
		Then 
		\begin{equation*}
			q \le r \ll 
			(\log P)^{2 C_d}, \qquad
			|q \theta - a| \le 
			|r \theta - 
			\lam(D) b| \ll 
			\frac{\lam(D) (\log P)^{2C_d}} {P^d}.
		\end{equation*}
		Put $\bet = \theta - a/q$. Applying Lemma \ref{lem4.3} with 
		$(f,G,b,m,Q)
		= (h, 1, r_D, D, \lambda(D))$ 
		gives
		\[
		\sum_{\substack{p \le P \\ p \equiv r_D \mmod D}} e \left(
		\frac{\theta h(p)}{\lam(D)} \right) \log p = I(\bet) \frac{S(q,a;D)}{\varphi(Dq)} + 
		O_h(Pe^{-c \sqrt{\log P}}),
		\]
		where
		\[
		I(\bet) = \int_2^P e\left(
		\frac{\bet h(t)}{\lam(D)}\right) \d t,
		\qquad
		S(q,a;D) = \sum_{\substack{t \mmod{Dq} \\ (t,q) = 1 \\ t \equiv r_D \mmod D}} e \left(
		\frac{a h(t)}
		{q\lam(D)}
		\right).
		\]
		Thus, by (\ref{eqn9.3}), we have
		\begin{equation}
			\label{eqn9.4}
			|S(q,a;D) I(\bet)|
			\gg \frac{P \varphi(Dq)}{L \varphi(D)}.
		\end{equation}
		
		In light of the trivial bound $|I(\beta)|\leqslant P$, we now have
		\[
		|S(q,a;D)| \gg \frac{\varphi(Dq)}{L \varphi(D)} \ge \frac{\varphi(q)}{L}.
		\]
		By \cite[Lemma 28]{Luc2006}, the GCD of the non-constant coefficients of
		\[
		h_D(x) := \frac{h(r_D + Dx)}{\lam(D)} \in \bZ[x]
		\]
		is $O_h(1)$. Now Lemma \ref{lem2.2} yields
		\begin{equation*}
			S(q,a;D) \ll_{h,\eps} q^{1+\eps-1/d},
		\end{equation*}
		and we conclude that
		\[
		q \ll_{h,\eps} L^{d+\eps}.
		\]
		It remains to bound 
		$\lVert q\theta\rVert$. 
		
		Observe that
		\begin{align*}
			\psi: (\bZ/Dq\bZ)^\times &\to (\bZ/D\bZ)^\times \\
			[t] &\mapsto [t]
		\end{align*}
		defines a surjective group homomorphism, and that $\psi^{-1}(r_D)$ is a coset of $\ker(\psi) \le (\bZ/Dq\bZ)^\times$. Therefore
		\begin{align*} |S(q,a;D)| &\le
			|\psi^{-1}(r_D)| = |\ker(\psi)| =
			\frac{\varphi(Dq)}{\varphi(D)}.
		\end{align*}
		Pairing this with \eqref{eqn9.4} yields
		\[
		|I(\bet)| \gg P/L.
		\]
		
		A change of variables gives
		\[
		I(\bet) = D \int_0^{P/D} 
		e(\bet h_D(z)) \d z + O(D),
		\]
		and now \cite[Theorem 7.3]{Vau1997} yields
		\[
		I(\bet) \ll_h (|\bet|/\lam(D))^{-1/d} + D.
		\]
		Consequently
		\[
		\bet \ll (L/P)^d \lam(D),
		\]
		and finally
		\[
		\| q\theta \| \ll
		q L^d \lam(D) /P^d.
		\]
	\end{proof}
	
	\subsection{Proof of Theorem \ref{thm9.2}}
	
	Let $K\in\N$. Writing $\cB=\cB_h(\balpha,\rho)$ and $\cC=\cC_h(\balpha,\rho)$, we deduce from the Siegel--Walfisz theorem (Theorem \ref{thm2.1}) that
	\[
	\sum_{p \in \cB} \log p + \sum_{p \in \cC} \log p = \sum_{\substack{p \le P \\
			p \equiv r_D \mmod D}}\log p \asymp \frac{P}{\varphi(D)}.
	\]
	Suppose $\balpha\in\T^K$ is such that the sum over $p\in\cC$ in the above is larger that than the sum over $p\in\cB$. If no such $\balpha$ were to exist, then Theorem \ref{thm9.2} would hold with $\Delta(h,K,\rho)\gg 1$. 
	For this choice of $\balpha$, we infer from Lemma \ref{lem9.3} the existence of some $\mathbf{m}\in\Z^K$ with $0<\lVert \mathbf{m}\rVert_{\infty}\leqslant K\rho^{-1}$ and some $L\ll_K \rho^{-K}$ such that (\ref{eqn9.3}) holds with $\theta=\mathbf{m}\cdot\balpha$. By increasing the value of $L$ if necessary---which only weakens the bound (\ref{eqn9.3}) that we have obtained---we may assume that $L = C\rho^{-K}$ for some suitably large constant $C=C(h,K)>K$ to be specified later. Applying Lemma \ref{lem9.4} supplies us with some $q\in\N$ such that
	\begin{equation}
		\label{eqn9.5}
		q \ll_{h,\eps} L^{d+\eps},
		\qquad \| q\mathbf{m} \cdot\balpha \| \ll \frac{qL^d \lam(D)}
		{ P^d}. 
	\end{equation}
	
	\bigskip
	
	To complete the proof, beginning from the above deductions, we proceed by induction on $K$. First, suppose that $K=1$ and write $m\in\Z$ in place of $\mathbf{m}\in\Z^K$. 
	As $h$ is intersective of the second kind, any integer $n \equiv r_{qmD} \mmod{qmD}$ satisfies
	\[
	h(n) \equiv 0 \mmod {\lam(qmD)},
	\qquad (n,qmD) = 1.
	\]
	Further, by the Siegel--Walfisz theorem and (\ref{eqn9.5}), we have
	\[
	\sum_{\substack{p \le P/L^2 
			\\ p \equiv r_{qmD}
			\mmod{qmD}}} \log p \asymp
	\frac{P/L^2}{\varphi(qmD)}
	\gg_{h,\eps} \frac{P}{L^{d+3+\eps} \varphi(D)}.
	\]
	Since $r_{qmD} \equiv r_D \mmod D$, every prime $p$ appearing in the sum on the left is congruent to $r_D$ modulo $D$. Using (\ref{eqn2.5}), we have $qm \lam(D) \mid \lam(qm) \lam(D) = \lam(qmD) \mid h(p)$ for each prime $p$ in the sum, whence
	\[
	\left\| \frac{\alp h(p)}{\lam(D)} \right\|
	\le \left|\frac{h(p)}
	{qm \lam(D)}\right| \| q m \alp \|
	\ll_h \frac{(P/L^2)^d}{q \lam(D)}
	\frac{q L^d \lam(D)}{P^d}
	= L^{-d}
	= (\rho/C)^{d}.
	\]
	Taking $C$ sufficiently large, we deduce that $p \in \cB$. We conclude that
	\[
	\sum_{p\in\cB}\log p \geqslant
	\sum_{\substack{p \le P/L^2 
			\\ p \equiv r_{qmD}
			\mmod{qmD}}} \log p
	\gg_{h,\eps} \frac{P}{L^{d+3+\eps} \varphi(D)} \gg_h \frac{\rho^{d+3+\eps}P}{\varphi(D)},
	\]
	as required.
	
	\bigskip
	
	Now suppose $K\geqslant 2$ and assume the induction hypothesis that Theorem \ref{thm9.2} holds with $K-1$ in place of $K$. Write  $\mathbf{m}=(\mathbf{m}',m_K)$ and $\balpha=(\balpha',\alpha_K)$. The induction hypothesis, applied to
	\[
	\frac{\lam(q m_K)}{m_K} \balp'
	\]
	informs us that the set
	\[
	\cA = \left \{
	p \le P/L^2: p \equiv r_{qm_K D} \mmod{qm_K D}, \;
	\left \|
	\frac{h(p)}{m_K \lam(D)} \balp'
	\right \| < \frac{\rho^2}{2K^2} \right \}
	\]
	satisfies
	\begin{align*}
		\sum_{p\in\cA}\log p &\geqslant \frac{P}{L^2\varphi(qm_K D)}
		\Delta
		\left(h,K-1; \frac{\rho^2}{2K^2}\right) \\
		&\gg_{h,K,\eps} 
		\frac{P}{\varphi(D)}
		\rho^{3 + K(d+\eps)} \Delta
		\left(h,K-1; \frac{\rho^2}{2K^2}\right).
	\end{align*}
	
	Let $a\in\Z$ be such that
	\begin{align*}
		\lVert q\mathbf{m}\cdot\balpha\rVert = |q\mathbf{m}\cdot
		\balpha - a|.
	\end{align*}
	For each $p\in\cA$, we have $h(p) \ll_h (P/L^2)^d$ so, by (\ref{eqn9.5}),
	\[
	\left| \alp_K \frac{h(p)}{\lam(D)} - \frac{a/q - \mathbf{m}' \cdot \balp'}{m_K} \frac{h(p)}{\lam(D)}
	\right| = \left\lvert \frac{h(p)}{m_K\lambda(D)}\right\rvert\cdot \lvert\mathbf{m}\cdot\balpha - a/q\rvert
	\ll_h L^{-d}.
	\]
	As in the previous case, we have $qm_K\lambda(D) \mid h(p)$, whence
	\begin{equation*}
		\left\lVert \frac{a/q - \mathbf{m}' \cdot \balp'}{m_K} \frac{h(p)}{\lam(D)}\right\rVert = \left\lVert \frac{h(p)\mathbf{m}' \cdot \balp'}{m_K\lam(D)}\right\rVert
		< \frac{\rho}{2}.
	\end{equation*}
	Thus, by the triangle inequality,
	\[
	\left \| 
	\alp_K \frac{h(p)}{\lam(D)}
	\right \| 
	< \frac{\rho}2 +
	O_h(L^{-d}) = 
	\frac{\rho}2 + O_h((C\rho^{-K})^{-d}).
	\]
	By taking $C$ sufficiently large, we find that $p\in\cB$. Therefore
	\begin{equation*}
		\sum_{p\in\cB}\log p \geqslant \sum_{p\in\cA}\log p \gg_{h,K,\eps} \frac{P}{\varphi(D)}
		\rho^{3 + K(d+\eps)}
		\Delta
		\left(h,K-1; \frac{\rho^2}{2K^2}\right),
	\end{equation*}
	as required.
	
	\providecommand{\bysame}{\leavevmode\hbox to3em{\hrulefill}\thinspace}
	
\end{document}